\newif\ifdraft
\drafttrue
\documentclass[5p, a4paper]{article}

\usepackage[utf8]{inputenc}
\usepackage[T1]{fontenc}
\usepackage{amssymb}
\usepackage{amsfonts}
\usepackage{amsmath}
\usepackage{amsthm}
\usepackage{graphicx}
\usepackage{tikz}
\usepackage{pgfplots}
\usepackage{subfig} 
\usepackage{textcomp}
\usepackage{mathtools}
\usepackage{fullpage}

\usepackage{moreverb}
\usepackage{enumerate}

\newcommand{\R}{\mathbb{R}}
\newcommand{\N}{\mathbb{N}}

\newcommand{\eremk}{\hbox{}\hfill\rule{0.8ex}{0.8ex}}

\newtheorem{definition}{Definition}[section]
\newtheorem{theorem}[definition]{Theorem}
\newtheorem{lemma}[definition]{Lemma}
\newtheorem{proposition}[definition]{Proposition}
\newtheorem{remark}[definition]{Remark}
\newtheorem{algorithm}[definition]{Algorithm}

\numberwithin{equation}{section}

\ifdraft
\usepackage{color}
\newcommand{\todo}[1]{\textcolor{red}{#1}}
\else
\newcommand{\todo}[1]{}
\fi

\newcommand{\diam}{\operatorname*{diam}}

\newcommand{\abs}[1]{\left\vert #1 \right\vert}
\newcommand{\skp}[1]{\left< #1 \right>}
\newcommand{\norm}[1]{\left\| #1 \right\|}
\newcommand{\ra}[0]{\rightarrow}
\newcommand{\T}{\mathcal{T}}
\newcommand{\dist}{\operatorname*{dist}}
\renewcommand{\div}{\operatorname*{div}}
\newcommand{\supp}{\operatorname*{supp}}

\def\Crel{C_{\rm rel}}
\def\Ceff{C_{\rm eff}}
\def\Cmark{C_{\rm mark}}
\def\Clin{C_{\rm lin}}
\def\Copt{C_{\rm opt}}
\def\copt{c_{\rm opt}}
\def\qlin{q_{\rm lin}}
\def\opt{{\rm opt}}

\def\set#1#2{\big\{#1 \,:\, #2\big\}}
\def\TT{\T}
\def\MM{\mathcal{M}}
\def\interior{\operatorname*{interior}}
\def\patch{\Omega}

\def\refine{{\rm refine}}

\def\Cinv{C_{\rm inv}}

\def\UU{\mathcal{U}}
\def\Cstab{C_{\rm stab}}
\def\qred{q_{\rm red}}
\def\Cdrl{C_{\rm rel}}
\def\coarse{\bullet}
\def\fine{\circ}
\def\eps{\varepsilon}
\def\reff#1#2{\stackrel{\eqref{#1}}{#2}}
\def\NN{\mathcal{N}}
\def\Csz{C_{\rm SZ}}

\newcommand\blfootnote[1]{%
  \begingroup
  \renewcommand\thefootnote{}\footnote{#1}%
  \addtocounter{footnote}{-1}%
  \endgroup
}

\begin{document}

\hskip 10 pt

\begin{center}
{\fontsize{14}{20}\bf 
Quasi-optimal convergence rate for an adaptive method for the integral fractional Laplacian}
\end{center}

\begin{center}
\textbf{Markus Faustmann, Jens Markus Melenk, Dirk Praetorius}\\
\bigskip
{Institute for Analysis and Scientific Computing}\\
TU Wien\\
Wiedner Hauptstr. 8-10, 1040 Wien, Austria\\markus.faustmann@tuwien.ac.at, melenk@tuwien.ac.at, dirk.praetorius@asc.tuwien.ac.at\\
\bigskip
\end{center}

\begin{abstract}
For the discretization of the integral fractional Laplacian $(-\Delta)^s$, $0 < s < 1$, based on piecewise 
linear functions, we present and analyze a reliable weighted residual {\sl a posteriori} error estimator. 
In order to compensate for a lack of $L^2$-regularity of the residual in the regime $3/4 < s < 1$, 
this weighted residual error estimator includes as an additional weight a power of the 
distance from the mesh skeleton. We prove optimal convergence rates for an 
$h$-adaptive algorithm driven by this error estimator. Key to the analysis of the adaptive algorithm are 
local inverse estimates for the fractional Laplacian.  

\blfootnote{{\bf Acknowledgement.} The research of JMM and DP is funded by the Austrian Science Fund (FWF) 
by the special research program {\it Taming complexity in PDE systems} (grant SFB F65). 
Additionally, DP acknowledges support through the FWF research project 
{\it Optimal adaptivity for BEM and FEM-BEM coupling} (grant P27005).
}
%Classification: 26A33, 65N12, 65N30, 65N50
\end{abstract}

\section{Introduction}
\medskip
Fractional differential operators such as the fractional Laplacian $(-\Delta)^s$, $0 < s < 1$, 
are an increasingly important modeling tool in, e.g.,  physics, finance, or image processing. 
%In recent times, many applications in physics, finance, or image processing have started to employ models 
%that involve fractional powers of differential operators such as the fractional Laplacian 
%$(-\Delta)^s$ for $0 < s < 1$. 
In contrast to classical (integer order) differential operators, these fractional operators 
are nonlocal, which makes both the analysis and the analysis of numerical methods challenging. 
%sharing some to the behavior of the classical boundary integral operators in the 
%boundary element method (BEM); see~\cite{mclean,SauterSchwab}.

The numerical treatment of such non-local operators is currently a very active research field. 
A variety of approaches for the different versions of the fractional Laplacian in multi-d, or, more generally, 
fractional powers of differential operators are available: we mention 
Galerkin/finite element methods (FEMs) (see, e.g., \cite{NOS15,AcoBor17,ABH18,BMNOSS19} and references therein), 
techniques that exploit the connection of the fractional Laplacian with semigroup theory, 
\cite{bonito-pasciak15,bonito-pasciak17}, and techniques that rely on the connection with eigenfunction expansions, 
\cite{song-xu-karniadakis17,ainsworth-glusa18}; for more details, we refer the reader to the recent surveys 
\cite{BBNOS17,2018arXiv180109767L}. The vast majority of the numerical analysis literature focuses on {\sl a priori} error
analyses, and few results on {\sl a posteriori} error analysis are available in spite of the fact that solutions
to fractional differential equations typically have singularities (even for smooth
input data), which naturally calls for using locally refined meshes; work on {\sl a posteriori} error estimation
in a Galerkin setting includes 
\cite{nochetto-von-petersdorff-zhang10,CNOS15,AinGlu17,BBNOS17,CNOS15}
and on gradient recovery \cite{zhao-hu-cai-karniadakis17}.
One challenge in devising {\sl a posteriori} error estimators 
are poor properties of the residual, namely, it is not necessarily in $L^2$. One can overcome this lack of 
$L^2$-regularity by measuring the residual in appropriate $L^p$-spaces, 
\cite{nochetto-von-petersdorff-zhang10, BBNOS17} (some restrictions on $s$ apply); an alternative route, which is taken in
the present work, is to measure the residual in weighted $L^2$-spaces, where the weight is given by a power 
of the distance from the mesh skeleton.  The resulting {\sl a posteriori} error estimator is shown to be reliable in 
Theorem~\ref{theorem:reliable} in the full range $0 < s < 1$. This {\sl a posteriori} error estimator 
is a basic building block of the adaptive Algorithm~\ref{algorithm} that we analyze. We show 
in Theorem~\ref{theorem:algorithm} that it yields a sequence of approximations that converge at the 
optimal algebraic rate (with respect to an appropriate nonlinear approximation class). 
Such an optimal convergence result is well-known for adaptive FEMs for linear second-order elliptic equations 
(see, e.g., \cite{BDD04,stevenson07,ckns08,ffp14}) or the classical BEM (see~\cite{gantumur13,FKMP13,part1,part2}), 
and the present work extends these results to the integral fractional Laplacian. Our convergence result is obtained using 
the abstract, general framework of \cite{CFPP14} for proving such optimal convergence results,  
which reduces the proof of optimal algebraic convergence to verifying four properties (\ref{axiom:stability})---(\ref{axiom:quasiorthogonality})  of the 
error estimator (cf.\ Proposition~\ref{prop:axioms}). 
A key ingredient to establish the properties (A1)---(A4) in the case of the BEM are local inverse estimates for 
the nonlocal boundary integral operators, \cite{gantumur13,FKMP13,AFFKMP17}. Also in the present case of the 
fractional Laplacian underlying our analysis is the inverse estimate 
\begin{align}\label{eq:investintro}
\norm{\widetilde h_\ell^s(-\Delta)^s v_\ell}_{L^2(\Omega)} \leq C \norm{v_\ell}_{H^s(\Omega)},
\end{align}
where $v_\ell$ is a piecewise polynomial and $\widetilde h_\ell$ is the local mesh width function 
(modified by some additional weight function for $s \ge 1/2$; cf.~\eqref{eq:estimator}), and
is essentially sufficient for proving optimal convergence of the adaptive algorithm. In this work,
such an inverse estimate for the nonlocal fractional Laplacian is provided in Theorem~\ref{theorem:invest}. We 
highlight that such an inverse estimate proves useful in other applications such as the analysis of 
multilevel preconditioning for the fractional Laplacian on locally refined meshes. 

The present paper is structured as follows: In Section~\ref{sec:mainresults}, we provide the 
continuous model problem and its Galerkin discretization by piecewise linears. 
Moreover, both the classical weighted residual \emph{a posteriori} error indicators (for $0 < s<1/2$)
and our modified weighted error estimator (for $1/2 \le s < 1$) are presented. The adaptive algorithm,
the optimal convergence of the algorithm, as well as the inverse inequality, which plays the key role 
in our proofs, are stated. In Section~\ref{sec:algorithm}, the four essential properties (A1)---(A4) of the 
error estimator from \cite{CFPP14} are recalled and verified with the aid of the inverse inequality. 
The inverse inequality is then proved in Section~\ref{sec:invest}. Finally, Section~\ref{sec:numerics} 
provides numerical examples that illustrate the optimal convergence of the proposed adaptive method.

Concerning notation: For bounded, open sets $\omega \subset \R^d$ integer order Sobolev 
spaces $H^t(\omega)$, $t \in \N_0$, are  defined in the usual way. For $t \in (0,1)$, 
fractional Sobolev spaces are given in terms of the seminorm 
$|\cdot|_{H^t(\omega)}$ and the full norm $\|\cdot\|_{H^t(\omega)}$ by 
\begin{equation}
\label{eq:norm} 
|v|^2_{H^t(\omega)} = \int_{x \in \omega} \int_{y \in \omega} \frac{|v(x) - v(y)|^2}{\abs{x-y}^{d+2t}}
\,dx\,dy, 
\qquad \|v\|^2_{H^t(\omega)} = \|v\|^2_{L^2(\omega)} + |v|^2_{H^t(\omega)}, 
\end{equation}
where we denote the Euclidean distance in $\R^d$ by $\abs{\;\cdot\;}$. 
Moreover, for bounded Lipschitz domains $\Omega \subset \R^d$, we define the spaces
\begin{align*}
 \widetilde{H}^{t}(\Omega) := \{u \in H^t(\R^d) \,: \, u\equiv 0 \; \text{on} \; \R^d \backslash \Omega \}
\end{align*}
of $H^t$-functions with zero extension, equipped with the norm
\begin{align*}
 \norm{v}_{\widetilde{H}^{t}(\Omega)}^2 := \norm{v}_{H^t(\Omega)}^2 + \norm{v/\rho^t}_{L^2(\Omega)}^2,
\end{align*}
where $\rho(x)$ is the distance of a point $x \in \Omega$ to the boundary $\partial \Omega$.

%==========================================================================================
%==========================================================================================
\section{Main Results}\label{sec:mainresults}
%==========================================================================================
%==========================================================================================

\subsection{The fractional Laplacian and the Caffarelli-Silvestre extension}
There are several different ways to define the fractional Laplacian $(-\Delta)^s$. A classical 
definition on the full space ${\mathbb R}^d$ is in terms of the Fourier transformation ${\mathcal F}$, 
i.e., $({\mathcal F} (-\Delta)^s u)(\xi) = |\xi|^{2s} ({\mathcal F} u)(\xi)$. 
A consequence of 
this definition is the mapping property, (see, e.g., \cite{BBNOS17})
\begin{equation}
\label{eq:mapping-property}
(-\Delta)^s: H^t(\R^d) \rightarrow H^{t-2s}(\R^d), \qquad t \ge s, 
\end{equation}
where the Sobolev spaces $H^t(\R^d)$, $t \in \R$, are defined in terms of the Fourier transformation,
\cite[(3.21)]{mclean00}. 
Alternative, equivalent definitions of $(-\Delta)^s$ exist, e.g., via semi-group or 
operator theory,~\cite{Kwasnicki}. For our purposes, a convenient representation of the 
fractional Laplacian is as the principal value integral  
\begin{align}\label{eq:fracLaplaceDef}
(-\Delta)^su(x) := C(d,s) \; \text{P.V.} \int_{\R^d}\frac{u(x)-u(y)}{\abs{x-y}^{d+2s}} \, dy \quad \text{with} \quad
%C(d,s):= 2^{2s}s\frac{\Gamma(s+d/2)}{\pi^{d/2}\Gamma(1-s)},
C(d,s):= - 2^{2s}\frac{\Gamma(s+d/2)}{\pi^{d/2}\Gamma(-s)},
\end{align}
where $\Gamma(\cdot)$ denotes the Gamma function. 
An important observation made in~\cite{CafSil07} is that the fractional Laplacian can be understood as 
a Dirichlet-to-Neumann operator of a degenerate elliptic PDE, the so-called extension problem 
on a half space. In order to describe this degenerated elliptic PDE, we need weighted Sobolev spaces.
For 
\begin{equation}
\label{eq:alpha} 
\alpha :=1-2s \in (-1,1)
\end{equation} 
 and measurable subsets $\omega \subset \R^d \times \R^+$, we 
define the weighted $L^2$-norm 
\begin{align*}
 \norm{U}_{L^2_\alpha(\omega)}^2 := \int_{\omega} \mathcal{Y}^{\alpha}\abs{U(x,\mathcal{Y})}^2 dx \, d\mathcal{Y} 
\end{align*}
and denote by $L^2_\alpha(\omega)$ the space of 
%functions satisfying
%$\norm{U}_{L^2_\alpha(\omega)} < \infty$, i.e., $L^2_\alpha(\omega)$ is the space of 
square-integrable functions with respect to the weight $\mathcal{Y}^\alpha$. 
The Caffarelli-Silvestre extension is conveniently described 
in terms of the Beppo-Levi space 
${\mathcal B}^1_{\alpha}(\R^d \times  \R^+):= \{U \in {\mathcal D}^\prime(\R^{d} \times \R^+)\,|\, 
\nabla U \in L^2_\alpha(\R^d \times \R^+)\}$. 
Elements of ${\mathcal B}^1_\alpha(\R^d \times \R^+)$ are  in fact in $L^2_{\rm loc}(\R^d \times \R^+)$
and
one can give meaning to their trace 
at $\mathcal{Y} = 0$, which is denoted $\operatorname{tr} U$. Recalling $\alpha = 1-2s$, one has in fact
$\operatorname{tr} U \in H^s_{\rm loc}(\R^d)$ (see, e.g., \cite{KarMel18}). 
% 
%Then, the space $H_{\alpha}^1(\Omega\times \R^+)$ is defined as 
%\begin{align*}
%H_{\alpha}^1(\Omega\times \R^+):=\{U \in L^2_\alpha(\Omega \times \R^+) \; :\; 
%\nabla U \in L^2_\alpha(\Omega \times \R^+) \}.
%\end{align*}
The extension problem by Caffarelli-Silvestre~\cite{CafSil07} reads as follows:
For $u \in  \widetilde{H}^{s}(\Omega)$ and $\alpha = 1-2s$, 
let $U \in {\mathcal B}_{\alpha}^1(\R^d\times \R^+)$ solve 
\begin{subequations}\label{eq:extension}
\begin{align}
\label{eq:extension-a}
 \div (\mathcal{Y}^\alpha \nabla U) &= 0  \;\quad\quad\text{in} \; \R^d \times (0,\infty), \\ 
\label{eq:extension-b}
 \operatorname{tr} U & = u  \,\quad\quad\text{in} \; \R^d.
\end{align}
\end{subequations}
Then, the fractional Laplacian can be recovered as the Neumann data of the extension problem in the 
sense of distributions, \cite[Thm.~{3.1}]{cabre-sire14}:
\begin{align}
 -d_s \lim_{\mathcal{Y}\ra 0^+} \mathcal{Y}^\alpha \partial_\mathcal{Y} U(x,\mathcal{Y}) = (-\Delta)^s u, 
 %\lim_{y\ra 0^+} y^\alpha \partial_y U(x,y) = -c_s (-\Delta)^s u, 
\qquad 
%c_s = s2^{1-2s}\abs{\Gamma(-s)}/\Gamma(s) > 0.
d_s = 2^{1-2s}\abs{\Gamma(s)}/\Gamma(1-s). 
\end{align}
\begin{remark}
For $u \in H^s(\R^d)$, the extension $U$ given by \eqref{eq:extension} can also be characterized as \linebreak
$U = \operatorname{arg min} \{V \in {\mathcal B}^1_\alpha(\R^d \times \R^+)\,:\, 
\operatorname{tr} V = u\}$. 
\eremk
\end{remark}

%----------------------------------
\subsection{The model problem}
%----------------------------------
For a bounded Lipschitz domain $\Omega \subset \R^d$, we consider the problem 
\begin{subequations}\label{eq:modelproblem}
\begin{align}
 (-\Delta)^su &= f \qquad \text{in}\, \Omega, \\
 u &= 0 \quad \quad\, \text{in}\, \Omega^c:=\R^d \times \overline{\Omega},
\end{align}
\end{subequations}
where $f$ is a given right-hand side. The fractional Laplacian $(-\Delta)^su$ is defined by 
formula~\eqref{eq:fracLaplaceDef} for $x \in \Omega$.   
The weak formulation of~\eqref{eq:modelproblem} reads as follows~\cite[Thm. 1.1 (e),(g)]{Kwasnicki}: Find $u \in \widetilde{H}^s(\Omega)$ such that
\begin{align}\label{eq:weakform}
 a(u,v) %&:= \skp{u,v}_{H^s(\R^d)} \nonumber \\ &
 := \frac{C(d,s)}{2} \int\int_{\R^d\times\R^d} 
 \frac{(u(x)-u(y))(v(x)-v(y))}{\abs{x-y}^{d+2s}} \, dx \, dy = \int_{\Omega} fv \,dx 
 \quad \text{for all } v \in \widetilde{H}^s(\Omega).
\end{align}
Existence and uniqueness of $u \in \widetilde{H}^s(\Omega)$ follow from
the Lax--Milgram lemma.

%==========================================================================================
\subsection{Discretization}
%==========================================================================================
Henceforth, we assume $\Omega$ to be polyhedral\footnote{This is not essential but allows us to  
work with affine elements.}.  
Let $\T_\ell$ be a regular (in the sense of Ciarlet) triangulation of $\Omega$ consisting of open simplices that is $\gamma$-shape regular in the sense of $\max_{T \in \T_\ell} \big( \diam(T) / |T|^{1/d} \big) \le \gamma < \infty$. 
Here, $\diam(T):=\sup_{x,y\in T}|x-y|$ denotes the Euclidean diameter of $T$, whereas $|T|$ is the $d$-dimensional Lebesgue volume. To ease notation, we define the piecewise constant mesh size function $h_\ell \in L^\infty(\Omega)$ by
\begin{align}
 h_\ell|_T := h_\ell(T) := |T|^{1/d}
 \quad \text{for all } T \in \T_\ell.
\end{align}
Moreover, for all elements $T \in \T_\ell$, we define the element patch
\begin{align}\label{eq:patch}
 \patch_\ell(T) := \interior\left(\bigcup_{T' \in \TT_\ell(T)} \overline{T'} \right) \subseteq \Omega,
 \quad \text{where} \quad
 \TT_\ell(T) := \set{T' \in \TT_\ell}{\overline{T} \cap \overline{T'} \neq \emptyset}
\end{align}
consists of $T$ and all of its neighbors. 
$\patch_\ell^k(T) := {\rm interior}\big(\bigcup\set{T' \in \TT_\ell}{\overline{T'} \cap 
\overline{\patch_\ell^{k-1}(T)}\ne \emptyset}\big)$ 
is the $k$-th order patch of $T$ (and $\patch_\ell^1(T) := \patch_\ell(T)$).
Later on, the index $\ell \in \N_0$ indicates the step of an adaptive algorithm, where the triangulations $\TT_\ell$ are obtained by local mesh refinement based on newest vertex bisection (NVB). We refer, e.g., to~\cite{kpp13} for the NVB algorithm in $d = 2$ or to~\cite{stevenson08} for $d \ge 2$.

%For $k \ge 0$, we denote the space of $\TT_\ell$-piecewise polynomials of degree $k$ by
%\begin{align*}
% S^{k,0}(\T_\ell) := \{v \in L^2(\Omega) \,:\, v|_T \text{ is a polynomial of degree $k$ for all $T \in \T_\ell$}\}.
%\end{align*}
For $T \subseteq \Omega$, let $\mathcal{P}^1(T)$ denote the space of all affine functions on $T$.
We define spaces of $\TT_\ell$-piecewise affine and globally continuous functions
\begin{align}
S^{1,1}(\T_\ell) := %:= S^{1,0}(\T_\ell) \cap C(\overline{\Omega}) \subset H^1(\Omega)
\{u \in H^1(\Omega) \,:\, u|_{T} \in \mathcal{P}^1(T) \text{ for all } T \in \T_\ell\}
\quad \text{and} \quad
S^{1,1}_0(\T_\ell) := S^{1,1}(\T_\ell) \cap H_0^1(\Omega).
\end{align}%
For the discretization of~\eqref{eq:weakform}, we consider the Galerkin method based on 
$S^{1,1}_0(\T_\ell)$. The Lax--Milgram lemma provides existence and uniqueness of 
$u_\ell \in S^{1,1}_0(\T_\ell)$ such that
\begin{align}\label{eq:galerkin}
 a(u_\ell, v_\ell) = \int_{\Omega} fv_\ell \, dx 
 \quad \text{for all } v_\ell \in S^{1,1}_0(\T_\ell).
\end{align}
Throughout the present work, we will need a weight function that measures the distance 
from the mesh skeleton: For a mesh $\T_\ell$ we introduce
\begin{equation}
\label{eq:w_ell}
\omega_\ell(x):= \inf_{T\in \T_\ell} \dist(x,\partial T) 
= \inf_{T \in \T_\ell} \inf_{y \in \partial T} \abs{x - y}. 
\end{equation}
%
%==========================================================================================
\subsection{\textsl{A~posteriori} error estimation}
%==========================================================================================
Let $v_\ell \in S^{1,1}_0(\T_\ell)$. For $0 < s < 3/4$, due to the 
mapping properties of the fractional Laplacian given in~(\ref{eq:mapping-property}), 
we have $(-\Delta)^s v_\ell \in L^2(\Omega)$. 
For $3/4 < s < 1$, the function $(-\Delta)^s v_\ell$ is (generically) no longer in 
$L^2(\Omega)$ as it has singularities at the mesh skeleton. 
Its blow-up can be measured in terms of the distance from the mesh skeleton as has 
been noticed in \cite{BBNOS17} and is made precise in the following lemma, which is proved 
in Section~\ref{sec:invest} below. 

\begin{lemma}\label{lem:blowup}
%For $T \in \T_\ell$ and $x \in T$, let $\omega_\ell(x) := \dist(x,\partial T): = \inf_{y \in \partial T} \|x - y \|_2$. 
Let $\omega_\ell$ be given by (\ref{eq:w_ell}). 
Given $0 < s < 1$, fix $\beta > 2s-3/2$, e.g., $\beta := s-1/2$. 
For any $v_\ell \in S^{1,1}_0(\T_\ell)$, we then have $\omega_\ell^\beta (-\Delta)^s v_\ell \in L^2(\Omega)$.
%\dpr{Die Notation von Gewichtsfunktion und Elementpatch ist dieselbe!}
%\dpr{Zulässiger Bereich von $s$?}
\end{lemma}

To control the discretization error of~\eqref{eq:galerkin}, we study the weighted residual error estimator 
\begin{align}\label{eq:estimator}
 \eta_\ell(v_\ell) := \norm{\widetilde h_\ell^{s}\big(f-(-\Delta)^s v_\ell\big)}_{L^2(\Omega)},
 \,\, \text{where} \,\,\,
 \widetilde h_\ell^{s} :=
 \begin{cases}
   h_\ell^{s}\quad& \text{for $0 < s \le 1/2$},\\
   h_\ell^{s-\beta}\omega_\ell^\beta & \text{for $1/2 < s < 1$ and $\beta := s - 1/2$.}\\
 \end{cases}
\end{align}
Since the $L^2$-norm is local, the error estimator can be written as a sum of local contributions
\begin{align}\label{eq:indicators}
 \eta_\ell(v_\ell) = \left(\sum_{T\in\T_\ell} \eta_\ell(T,v_\ell)^2 \right)^{1/2}\!, 
 \quad \text{where} \quad
 \eta_\ell(T,v_\ell) := \norm{\widetilde h_\ell^{s}\big(f-(-\Delta)^s v_\ell\big)}_{L^2(T)}.
\end{align}
If $v_\ell = u_\ell$ is the solution of~\eqref{eq:galerkin}, we abbreviate $\eta_\ell := \eta_\ell(u_\ell)$ as well as $\eta_\ell(T) := \eta_\ell(T,u_\ell)$ for all $T \in \T_\ell$.

\begin{theorem}\label{theorem:reliable}
For $0 < s < 1$, the weighted residual error estimator~\eqref{eq:estimator} is reliable:  
\begin{align}\label{eq:reliable}
 \norm{u - u_\ell}_{\widetilde H^s(\Omega)}
 \le \Crel \, \eta_\ell.
\end{align}
Moreover, for $0 < s \leq 1/2$ with $s \neq 1/4$ and $u \in H^1_0(\Omega)$, the estimator is efficient in the sense that
\begin{align}\label{eq:efficient}
 \eta_\ell^2 
 \le \Ceff^2 \,\left(\norm{u - u_\ell}_{\widetilde{H}^s(\Omega)}^2 + 
 \sum_{T \in \T_\ell}h_\ell(T) \norm{u - u_\ell}_{H^{1/2+s}(\Omega^2_\ell(T))}^2 \right).
\end{align}
The constants $\Crel, \Ceff > 0$ depend only on $\Omega$, $d$, $s$, and the $\gamma$-shape regularity of $\TT_\ell$.
\end{theorem}%
\begin{remark}
The efficiency result \eqref{eq:efficient} is weaker than classical efficiency, where the sum on the right-hand side 
does not appear. However, as this term scales with the correct powers of the mesh width, there is not 
a large gap between \eqref{eq:efficient} and classical efficiency, as for discrete functions 
$u$ the right-hand side of \eqref{eq:efficient} can be bounded by the energy norm with an inverse estimate.
\eremk
\end{remark}

%==========================================================================================
\subsection{Adaptive mesh refinement}
%==========================================================================================

Based on the local contributions of the weighted residual error estimator~\eqref{eq:estimator}, we consider the following standard approach for adaptive mesh refinement of the type \texttt{SOLVE} -- \texttt{ESTIMATE} -- \texttt{MARK} -- \texttt{REFINE}, where the D\"orfler criterion~\eqref{eq:doerfler} from~\cite{Doer96} is used to select elements for refinement.

\begin{algorithm}\label{algorithm}
Input: Initial triangulation $\T_0$, adaptivity parameters $0 < \theta \le 1$, $\Cmark \ge 1$.
\newline
For all $\ell = 0, 1, 2, \dots$, iterate the following steps~(\ref{item:alg-i})--(\ref{item:alg-iv}):
\begin{enumerate}[(i)]
\item 
\label{item:alg-i}
Compute the solution $u_\ell \in S^{1,1}_0(\T_\ell)$ of~\eqref{eq:galerkin}.
\item 
\label{item:alg-ii}
Compute refinement indicators $\eta_\ell(T) = \eta_\ell(T,u_\ell)$ from~\eqref{eq:indicators} for all $T \in \T_\ell$.
\item 
\label{item:alg-iii}
Determine a set $\MM_\ell \subseteq \T_\ell$ of, up to the multiplicative factor $\Cmark$, minimal cardinality such that
\begin{align}\label{eq:doerfler}
 \theta \sum_{T\in\T_\ell} \eta_\ell(T)^2 \le \sum_{T\in\MM_\ell} \eta_\ell(T)^2.
\end{align}
\item 
\label{item:alg-iv}
Generate the coarsest NVB refinement $\T_{\ell+1} := \refine(\T_\ell, \MM_\ell)$ of $\TT_\ell$ such that all marked elements $T \in \T_\ell$ have been bisected. 
\end{enumerate}
\end{algorithm}

Analogous to the works~\cite{stevenson07,ckns08,ffp14,CFPP14} for FEM and BEM, the following 
Theorem~\ref{theorem:algorithm} states linear convergence~\eqref{eq:linear} of Algorithm~\ref{algorithm} with optimal algebraic convergence rates~\eqref{eq:optimal}.

\begin{theorem}\label{theorem:algorithm}
Let $0 < \theta \le1$ and $1 \le \Cmark \le \infty$. Then, there exist constants $0 < \qlin < 1$ and $\Clin > 0$ 
such that the sequence $(u_\ell)_{\ell \in \N}$ generated by Algorithm~\ref{algorithm} satisfies 
\begin{align}\label{eq:linear}
 \eta_{\ell + n} \le \Clin \qlin^n \, \eta_\ell
 \quad \text{for all } \ell, n \in \N_0.
\end{align}
Moreover, if $0 < \theta \ll 1$ is sufficiently small and  $1 \le \Cmark <\infty$, 
and if the initial triangulation $\T_0$ satisfies the admissibility condition~\cite[Section~4]{stevenson08} for $d \ge 3$, 
then the error estimator converges with the best possible algebraic rate: 
For each $t > 0$, there exist $\copt$, $\Copt > 0$ such that
\begin{align}\label{eq:optimal}
 & \copt \, \mathbb{A}_t(u) \le 
 \sup_{\ell \in \N_0} (\#\T_\ell)^t \, \eta_\ell
 \le \Copt \, \mathbb{A}_t(u), 
\\
&
\mathbb{A}_t(u) := \sup_{N \in \N_0} N^t \, \min_{\T_\opt \in \{\T \in \refine(\T_0) \,:\, \#\T - \#\T_0 \le N\}} \eta_\opt, 
\end{align}
where $\refine(\T_0)$ is the (infinite) set of all NVB refinements of the initial triangulation $\T_0$ 
and $\eta_\opt$ is the weighted residual error estimator corresponding to the triangulation $\T_\opt$.
%where $\mathbb{A}_t(u) := \sup_{N \in \N_0} N^t \, \min_{\T_\opt \in \{\T \in \refine(\T_0) \,:\, \#\T - \#\T_0 \le N\}} \eta_\opt$ with $\refine(\T_0)$ being the set of all NVB refinements of the initial triangulation $\T_0$ and $\eta_\opt$ being the weighted residual error estimator corresponding to the triangulation $\T_\opt$.
\end{theorem}

%==========================================================================================
\subsection{Inverse estimates for the nonlocal operator $\boldsymbol{(-\Delta)^s}$}
%==========================================================================================

A key piece in the axiomatic approach of \cite{CFPP14}, which generalizes ideas and 
techniques developed in \cite{FKMP13,gantumur13,part1,part2} for the convergence analysis 
of adaptive algorithms involving nonlocal operators, are inverse inequalities for the pertinent
operators. The following Theorem~\ref{theorem:invest} establishes this inverse estimate for the 
fractional Laplacian. 
%Following the ideas established in \cite{FKMP13,gantumur13,part1,part2}, we observe that 
%The missing piece in the axiomatic approach of \cite{CFPP14} is an inverse inequality for the fractional Laplacian, which is stated 
%in the following Theorem~\ref{theorem:invest}. 
We mention that the case $s = 1/2$ for the fractional Laplacian 
is closely related to the hyper-singular integral operator in the BEM~\cite{gantumur13,part2}, 
for which the appropriate inverse estimate is established in \cite{AFFKMP17}. 

\begin{theorem}\label{theorem:invest} 
\begin{enumerate}[(i)]
\item 
Let $0 < s \le 1/2$ with $s\neq 1/4$. Then, there holds for $v \in \widetilde H^{2s}(\Omega)$
\begin{align}\label{eq1:invest}
 \norm{h_\ell^{s} \, (-\Delta)^s v}_{L^2(\Omega)}
 \le \Cinv \, \Big( \norm{v}_{\widetilde H^s(\Omega)}^2 
 + \sum_{T\in\T_\ell }h_\ell(T)^{2s}\norm{ v}_{H^{2s}(\patch_\ell(T))}^2 \Big)^{1/2}.
\end{align}
\item 
For $0 < s < 1$ and all $v_\ell \in S^{1,1}_0(\T_\ell)$, there holds 
\begin{align}\label{eq:invest}
 \norm{\widetilde h_\ell^{s} \, (-\Delta)^s v_\ell}_{L^2(\Omega)} 
 \le \Cinv \, \norm{v_\ell}_{\widetilde H^s(\Omega)}.
\end{align}%
\end{enumerate}
The constant $\Cinv > 0$ depends only on $\Omega$, $d$, $s$, and the $\gamma$-shape regularity of $\T_\ell$. 
\end{theorem}

%==========================================================================================
%==========================================================================================
\section{Proof of Theorem~\ref{theorem:reliable} and Theorem~\ref{theorem:algorithm}}
\label{sec:algorithm}
%==========================================================================================
%==========================================================================================

%==========================================================================================
\subsection{Axioms of adaptivity}
%==========================================================================================

The following Proposition~\ref{prop:axioms} yields validity of the \emph{axioms of adaptivity} from~\cite{CFPP14}, where we recall that the present conforming discretization guarantees that $S^{1,1}_0(\T_\bullet) \subseteq S^{1,1}_0(\T_\circ)$ if the triangulation $\TT_\fine$ is a refinement of $\TT_\coarse$ (i.e., $\TT_\circ$ is obtained from $\TT_\bullet$ by finitely many steps of newest vertex bisection).

\begin{proposition}\label{prop:axioms}
There exist constants $\Cstab$, $\Cdrl>0$ and $0 < \qred < 1$ depending solely on $\Omega$, $d$, $s$, and the $\gamma$-shape regularity of the initial triangulation $\TT_0$ such that the following properties 
(\ref{axiom:stability})--(\ref{axiom:quasiorthogonality}) hold  for any 
 refinement $\TT_\coarse$ of the initial triangulation $\TT_0$ and any refinement $\TT_\fine$ of $\TT_\coarse$: 
%Then, there exist constants $\Cstab, \Cdrl > 0$ and $0 < \qred < 1$, which depend only on $\Omega$, $d$, $s$, and $\gamma$-shape regularity of $\TT_\coarse$ and $\TT_\fine$, such that the following properties~\eqref{axiom:stability}--\eqref{axiom:reliability} are satisfied:
\begin{enumerate}
\renewcommand{\theenumi}{A\arabic{enumi}}
\bf
\item[(A1)] Stability:\refstepcounter{enumi}\label{axiom:stability}
\rm
For all $v_\fine \in S^{1,1}_0(\T_\fine), w_\coarse \in S^{1,1}_0(\T_\coarse)$ and any $\UU_\coarse \subseteq \T_\coarse \cap \T_\fine$, 
there holds 
\begin{align*}
\left| \left( \sum_{T \in \UU_\coarse} \eta_\fine(T,v_\fine)^2 \right)^{1/2} - \left(\sum_{T \in \UU_\coarse} \eta_\coarse(T,w_\coarse)^2 \right)^{1/2} \right|
\le \Cstab \, \norm{v_\fine - w_\coarse}_{\widetilde{H}^s(\Omega)}.
\end{align*}

\bf
\item[(A2)] Reduction:\refstepcounter{enumi}\label{axiom:reduction}
\rm
For all $v_\coarse \in S^{1,1}(\TT_\coarse)$, there holds 
\begin{align*}
 \left( \sum_{T \in \TT_\fine \backslash \TT_\coarse} \eta_\fine(T,v_\coarse)^2 \right)^{1/2} 
 \le \qred \, \left( \sum_{T \in \TT_\coarse \backslash \TT_\fine} \eta_\coarse(T,v_\coarse)^2 \right)^{1/2}.
\end{align*}

\bf
\item[(A3)] Discrete reliability:\refstepcounter{enumi}\label{axiom:reliability}
\rm
The Galerkin approximations $u_\coarse \in S^{1,1}_0(\TT_\coarse)$ and $u_\fine \in S^{1,1}_0(\TT_\fine)$ satisfy that
\begin{align*}
 \norm{u_\fine - u_\coarse}_{\widetilde{H}^s(\Omega)} 
 \le \Cdrl \, \left( \sum_{T \in \T_\coarse \backslash \T_\fine} \eta_\coarse(T,u_\coarse)^2 \right)^{1/2}.
\end{align*}

\bf 
\item[(A4)] Quasi-orthogonality:\refstepcounter{enumi}\label{axiom:quasiorthogonality}
\rm 
For all $\ell$, $N>0$ and $\varepsilon>0$, it holds that
\begin{align*}
\sum_{k = \ell}^N \left(\norm{u_{k+1}-u_k}_{\widetilde{H}^s(\Omega)}^2 - \varepsilon \eta_k^2\right) \leq C_{\rm orth}(\varepsilon) \eta_\ell^2.
\end{align*}
\end{enumerate}
\end{proposition}

\begin{proof}[Proof of~\eqref{axiom:quasiorthogonality}]
The quasi-orthogonality 
follows directly from the fact that the bilinear form $a(\cdot,\cdot)$ of the fractional Laplacian~\eqref{eq:weakform} 
is bilinear, elliptic, and symmetric; see~\cite[Section~3.6]{CFPP14} for details. 
\end{proof}

\begin{proof}[Proof of~\eqref{axiom:stability}]
%Define $\widetilde h^s_\coarse := h_\coarse^s$ for $0 < s \le 1/2$ and $\widetilde h^s_\coarse := h_\coarse^{s-\beta}\omega_\coarse^\beta$ for $1/2 < s < 1$. 
Let $\omega := \interior \big( \bigcup_{T \in \UU_\bullet} \overline{T} \big) \subseteq \Omega$. With Theorem~\ref{theorem:invest}, we obtain that 
\begin{align*}
 \left| \left( \sum_{T \in \UU_\coarse} \eta_\fine(T,v_\fine)^2 \right)^{1/2} 
 	\!\!\!- \left(\sum_{T \in \UU_\coarse} \eta_\coarse(T,w_\coarse)^2 \right)^{1/2} \right|
 &= \left| \norm{\widetilde h^s_\coarse \big( f - (-\Delta)^s v_\fine \big)}_{L^2(\omega)}
 	-  \norm{\widetilde h^s_\coarse \big( f - (-\Delta)^s w_\coarse \big)}_{L^2(\omega)} \right|
 \\
 &\le \norm{\widetilde h^s_\coarse (-\Delta)^s (v_\fine - w_\coarse) }_{L^2(\omega)}
 \reff{eq:invest}\lesssim \norm{v_\fine - w_\coarse}_{\widetilde H^s(\Omega)}.
 %+ \norm{h_\coarse^{1-s}\nabla (v_\coarse - w_\coarse)}_{L^2(\Omega)}.
\end{align*}
%Finally, an inverse estimate from~\cite{ghs05,georgoulis08} (see also~\cite{affkp15} for the interpolation argument for $\widetilde{H}^s$-norms) proves that
%\begin{align*}
% \norm{h_\coarse^{1-s}\nabla (v_\coarse - w_\coarse)}_{L^2(\Omega)}
% \lesssim \norm{v_\coarse - w_\coarse}_{\widetilde H^s(\Omega)}.
%\end{align*}
%The hidden constants depend only on $\Omega$, $s$, and $\gamma$-shape regularity of $\T_\coarse$.
%This concludes the proof.
This proves~\eqref{axiom:stability} with $\Cstab = \Cinv$.
\end{proof}

\begin{proof}[Proof of~\eqref{axiom:reduction}]
Bisection ensures that $|T'| \le |T|/2$ for all $T \in \T_\coarse \backslash \T_\fine$ 
and its children/descendants $T' \in \T_\fine \backslash \T_\coarse$ with $T' \subset T$.
Note that $\omega := \bigcup_{T' \in \TT_\fine \backslash \TT_\coarse} \overline{T'} = \bigcup_{T \in \TT_\coarse \backslash \TT_\fine} \overline{T}$.
For $0 < s \le 1/2$, this proves~\eqref{axiom:reduction} with $\qred = 2^{-s/d}$, since
\begin{align*}
 &\left( \sum_{T \in \TT_\fine \backslash \TT_\coarse} \eta_\fine(T,v_\coarse)^2 \right)^{1/2}
 = \left( \sum_{T \in \TT_\fine \backslash \TT_\coarse} |T|^{2s/d} \, \norm{\big(f-(-\Delta)^s v_\coarse\big)}_{L^2(T)}^2 \right)^{1/2}
 \\& \qquad \qquad
 \le 2^{-s/d} \, \left( \sum_{T \in \TT_\coarse \backslash \TT_\fine} |T|^{2s/d} \, \norm{\big(f-(-\Delta)^s v_\coarse\big)}_{L^2(T)}^2 \right)^{1/2}
 = 2^{-s/d}  \left( \sum_{T \in \TT_\coarse \backslash \TT_\fine} \eta_\coarse(T,v_\coarse)^2 \right)^{1/2}.
\end{align*}
For $1/2 < s < 1$, we note that $0 < \beta = s - 1/2 < s$ and $s - \beta > 0$. Moreover, we have 
pointwise $\omega_\fine(x) \le \omega_\coarse(x)$. Hence, it follows that 
\begin{align*}
 h_\fine^{s-\beta} \omega_\fine^\beta \le 2^{-(s-\beta)/d} \, h_\coarse^{s-\beta} \omega_\coarse^\beta
 \quad \text{pointwise on } \omega \subseteq \Omega.
\end{align*}
Arguing as before, we prove~\eqref{axiom:reduction} with $\qred = 2^{-(s-\beta)/d}$. This concludes the proof.
\end{proof}

The proof of discrete reliability~\eqref{axiom:reliability} relies on the Scott--Zhang projection~\cite{ScottZhang} 
for quasi-interpolation in $\widetilde{H}^s(\Omega)$. 
While the original work~\cite{ScottZhang} is concerned with the integer-order Sobolev space $H^1(\Omega)$, 
the approach is generalized in~\cite{affkp15} to fractional-order Sobolev spaces 
$\widetilde{H}^s(\Omega)$ for $0 \le s \le 1$. 
Since the precise construction will matter, we briefly sketch it: 
Let $\NN_\coarse$ be the set of all nodes of $\T_\coarse$ and let 
$\NN_\coarse^{\rm int}$ be the set of nodes of $\T_\coarse$, which lie \emph{inside} $\Omega$ (and not on the boundary $\Gamma$). 
For each element $T \in \T_\coarse$, 
let $\{\phi^\star_{1,T},\ldots,\phi^\star_{d,T}\} \subset \mathcal{P}^1(T)$ 
be the dual basis for the nodal basis of $\mathcal{P}^1(T)$ with 
respect to $(\cdot,\cdot)_{L^2(T)}$. these functions may be viewed as 
elements of $S^{1,0}(\T_\coarse)$ by zero extension outside $T$. 
For each $z \in \NN_\coarse$ choose an \emph{arbitrary} element 
$T_z \in \T_\coarse$ with $z \in \overline{T_z}$. (This freedom will be 
exploited later on.) Let $\phi_z \in S^{1,1}(\T_\coarse)$ 
denote the hat function associated with $z$. 
For the element $T_z$, 
let $\phi^\star_z \in \{\phi^\star_{1,T_z},\ldots,\phi^\star_{d,T_z}\} \subset \mathcal{P}^1(T_z)$ 
be such that $\int_{T_z} \phi^\star_z \phi_{z'} \, dx = \delta_{zz'}$ 
for all $z' \in \NN_\coarse$. Then, the Scott--Zhang projector $J_\coarse : L^2(\Omega) \to S^{1,1}_0(\TT_\coarse)$ is defined by
\begin{align*}
 J_\coarse v := \sum_{z \in \NN_\coarse^{\rm int}} \left( \int_{T_z} \phi^\star_z v \, dx \right) \, \phi_z.
\end{align*} 
Since the summation is over the interior nodes $\NN_\coarse^{\rm int}$, we have 
$J_\coarse v \in S^{1,1}_0(\T_\coarse)$. 
The following lemma collects some of its properties:

\begin{lemma}\label{lemma:sz}
For $0 < s < 1$, there holds 
for some constant $\Csz > 0$ depending only on $s$ and the $\gamma$-shape regularity of $\TT_\bullet$ and $\Omega$: 
\begin{enumerate}[(i)]
\item 
\label{item:lemma:sz-i}
$J_\bullet V_\bullet = V_\bullet$ \quad as well as \quad $\norm{J_\bullet v}_{\widetilde{H}^s(\Omega)} \le \Csz \, \norm{v}_{\widetilde{H}^s(\Omega)}$ \quad for all $V_\bullet \in S^{1,1}_0(\T_\bullet)$ and all $v \in \widetilde{H}^s(\Omega)$.
\item
\label{item:lemma:sz-ii}
$h_\bullet(T)^{-1} \, \norm{(1-J_\bullet) v}_{L^2(T)} + \norm{\nabla(1-J_\bullet) v}_{L^2(T)} \le \Csz \, \norm{\nabla v}_{L^2(\Omega_\bullet(T))}$
\quad for all $v \in H^1_0(\Omega)$.
\item
\label{item:lemma:sz-iii}
$\norm{\widetilde h_\bullet^{-s} \, (1-J_\bullet) v}_{L^2(\Omega)} \le \Csz \, \norm{v}_{\widetilde{H}^s(\Omega)}$ \quad for all $v \in \widetilde{H}^s(\Omega)$.
\item 
\label{item:lemma:sz-iv}
For each refinement $\T_\fine$ of $\T_\coarse$, the Scott-Zhang operator $J_\coarse$ can be chosen such that 
additionally 
\begin{align}
\label{eq:property-SZ}
 (1-J_\coarse) v = 0
 \quad \text{on } \omega := \bigcup_{T \in \T_\coarse \cap \T_\fine} \overline{T} \quad
  \text{for all } v \in S^{1,0}(\T_\fine). 
\end{align}
\end{enumerate}
\end{lemma}

\begin{proof}
(\ref{item:lemma:sz-i})--(\ref{item:lemma:sz-ii}) are proved in~\cite[Lemma~4]{affkp15}. 
With $1/\widetilde h_\bullet^{s}$ replaced by $h_\bullet^{-s}$, (\ref{item:lemma:sz-iii}) is 
proved in~\cite[Lemma~3.2]{part2}, where the estimate
\begin{align}\label{eq:SZsmallS}
\norm{h_\bullet^{-t} \, (1-J_\bullet) v}_{L^2(\Omega)} 
\lesssim  \norm{v}_{\widetilde H^t(\Omega)}
\end{align}
is established for any $0< t < 1$ as well.
Hence, it only remains to prove~(\ref{item:lemma:sz-iii}) for $1/2 < s < 1$, where $\widetilde h_\bullet^{-s} = h_\bullet^{\beta-s}\omega_\bullet^{-\beta}$.

With $\widehat T$ being the reference element, let $\omega_{\widehat T}(\widehat x) := \dist(\widehat x,\partial\widehat T)$ be the corresponding weight function. 
Let $\Phi_T : \widehat T \to T$ be an affine parametrization of $T \in \TT_\bullet$. 
For each $\widehat x \in \widehat T$, let $x := \Phi_T(\widehat x)$. 
Let $x_{\rm min} \in \partial T$ with $\abs{x - x_{\rm min}} = \dist(x,\partial T)$. Then, 
\begin{align*}
 \abs{x - x_{\rm min}} \le \abs{x - \Phi_T(\widehat x_{\rm min})} \lesssim h_\bullet(T) \, \abs{\widehat x - \widehat x_{\rm min}}
 \le h_\bullet(T) \, \abs{\widehat x - \Phi_T^{-1}(x_{\rm min})} \lesssim \abs{x - x_{\rm min}}
\end{align*}
and consequently
\begin{align*}
 h_\bullet(T) \, \omega_{\widehat T} \circ \Phi_T \simeq \omega_\bullet|_T
 \quad \text{pointwise for all } T \in \TT_\bullet.
\end{align*}
According to~\cite[Thm.~{1.4.4.3}]{Grisvard}, we have $\norm{\omega_{\widehat T}^{-\beta} \widehat w}_{L^2(\widehat T)} \lesssim \norm{\widehat w}_{H^\beta(\widehat T)}$ for all $\widehat w \in H^\beta(\widehat T)$, since $0 < \beta < 1/2$. Let $w := (1-J_\bullet) v$ and $\widehat w := w \circ \Phi_T$. 
A scaling argument thus proves that
\begin{align*}
 \norm{\omega_\bullet^{-\beta} w}_{L^2(T)}
 &\simeq |T|^{1/2} \, h_\bullet(T)^{-\beta} \, \norm{\omega_{\widehat T}^{-\beta} \widehat w}_{L^2(\widehat T)} 
 \lesssim |T|^{1/2} \, h_\bullet(T)^{-\beta} \norm{\widehat w}_{H^\beta(\widehat T)}
 \stackrel{0<\beta <s}{\lesssim} |T|^{1/2} \, h_\bullet(T)^{-\beta} \norm{\widehat w}_{H^s(\widehat T)}\\
 &\lesssim h_\bullet(T)^{s-\beta} |w|_{H^s(T)}+h_\bullet(T)^{-\beta} \norm{w}_{L^2(T)}.
\end{align*} 
Hence, by summation over all elements 
\begin{align*}
\norm{\widetilde h^{-s} w}_{L^2(\Omega)}  
&= \norm{h_\bullet^{\beta-s} \omega_\bullet^{-\beta} w}_{L^2(\Omega)} 
\lesssim \abs{w}_{H^s(\Omega)} + \norm{h_\bullet^{-s} w}_{L^2(\Omega)} 
\stackrel{\eqref{eq:SZsmallS}}{\lesssim} 
\abs{w}_{\widetilde H^s(\Omega)} + \norm{v}_{\widetilde H^s(\Omega)} 
\stackrel{(\ref{item:lemma:sz-i})}{\lesssim} \norm{v}_{\widetilde H^s(\Omega)}. 
\end{align*}
This concludes the proof of (\ref{item:lemma:sz-iii}). 

We prove (\ref{item:lemma:sz-iv}). To ensure the additional property 
\eqref{eq:property-SZ}, 
we use the freedom still available in the definition of the 
Scott-Zhang operator $J_\coarse$.  
To that end, let $\NN_{\coarse\cap\fine}$ be the nodes of the elements 
in $\T_\coarse \cap \T_\fine$.  For each $z \in \NN_{\coarse\cap\fine} \subset \NN_\coarse$, 
select an element $T_z \in \T_\coarse \cap \T_\fine$ such $z$ is a node of $T_z$; 
for the remaining nodes $z \in \NN_\coarse \setminus \NN_{\coarse\cap\fine}$, the element 
$T_z \in \T_\coarse$ merely needs to be such that $z$ is a node of $T_z$. This defines 
the Scott-Zhang operator $J_\coarse$. This particular choice ensures 
\begin{equation}
\label{eq:SZ-10}
(J_\coarse v)(z) = v(z) \qquad \text{for all } z \in \NN_{\coarse\cap\fine} \quad \text{and all } v \in S^{1,0}(\T_\fine), 
\end{equation}
and \eqref{eq:property-SZ} follows from \eqref{eq:SZ-10}. 
\end{proof}

\begin{proof}[Proof of~\eqref{axiom:reliability}]
%Define $\widetilde h^s_\coarse := h_\coarse^s$ for $0 < s \le 1/2$ and $\widetilde h^s_\coarse := h_\coarse^{s-\beta}\omega_\coarse^\beta$ for $1/2 < s < 1$. 
Since $a(\cdot,\cdot)$ is elliptic on $\widetilde{H}^s(\Omega)$, we have 
\begin{align*}
 &\norm{u_\fine - u_\coarse}_{\widetilde H^s(\Omega)}^2
 \lesssim a(u_\fine - u_\coarse, u_\fine - u_\coarse)
 = a(u_\fine - u_\coarse, (1 - J_\coarse) (u_\fine - u_\coarse))
 \\&\qquad
 = \int_\Omega \big( f - (-\Delta)^s u_\coarse \big) \, (1 - J_\coarse)(u_\fine - u_\coarse) \, dx.
\end{align*}
%\JMM{For all nodes $z \in \NN_\coarse$ that belong to $\T_\coarse \cap \T_\fine$ we may choose elements
%$T_z$ in the definition of the Scott-Zhang projection such that $T_z \in \T_\coarse \cap \T_\fine$. }
%Hence, it follows that
%\begin{align*}
% (1-J_\coarse)(u_\fine - u_\coarse) = 0
% \quad \text{on } \omega := \bigcup_{T \in \T_\coarse \cap \T_\fine} \overline{T}.
%\end{align*}
In particular, the Cauchy--Schwarz inequality and Lemma~\ref{lemma:sz} (\ref{item:lemma:sz-iv}) show 
with $\omega:= \bigcup_{T \in \T_\coarse\cap\T_\fine} \overline{T}$
\begin{align*}
 \norm{u_\fine - u_\coarse}_{\widetilde H^s(\Omega)}^2
 &\lesssim \norm{\widetilde h_\coarse^{s} \, \big(f - (-\Delta)^s u_\coarse \big)}_{L^2(\Omega \backslash\omega)}
 \norm{\widetilde h_\coarse^{-s} \, (1 - J_\coarse)(u_\fine - u_\coarse)}_{L^2(\Omega \backslash\omega)}
 \\&
 = \left(\sum_{T \in \TT_\coarse\backslash\TT_\fine} \eta_\coarse(T,u_\coarse)^2 \right)^{1/2} \, 
 \norm{\widetilde h_\coarse^{-s} \, (1 - J_\coarse)(u_\fine - u_\coarse)}_{L^2(\Omega \backslash\omega)}.
\end{align*}
Lemma~\ref{lemma:sz} (\ref{item:lemma:sz-iii}) proves that
\begin{align*}
 \norm{\widetilde h_\coarse^{-s} \, (1 - J_\coarse)(u_\fine - u_\coarse)}_{L^2(\Omega)}
 \lesssim \norm{u_\fine - u_\coarse}_{\widetilde{H}^s(\Omega)}.
\end{align*}
The combination of the last two estimates concludes the proof.
\end{proof}

%==========================================================================================
\subsection{Proof of Theorem~\ref{theorem:reliable}}
%==========================================================================================

For reliability~\eqref{eq:reliable}, we sketch the proof from~\cite[Section~3.3]{CFPP14}.
Let $\TT_\coarse$ be a given triangulation. Recall that uniform refinement leads to convergence. Given $\eps > 0$, we may hence choose a refinement $\TT_\fine$ of $\TT_\coarse$ such that $\norm{u - u_\fine}_{\widetilde{H}^s(\Omega)} \le \eps$. Therefore, the triangle inequality and~\eqref{axiom:reliability} prove that
\begin{align*}
 \norm{u - u_\coarse}_{\widetilde{H}^s(\Omega)}
 \le \norm{u - u_\fine}_{\widetilde{H}^s(\Omega)} + \norm{u_\fine - u_\coarse}_{\widetilde{H}^s(\Omega)}
 \reff{axiom:reliability}\le \eps + \Cdrl \, \left( \sum_{T \in \T_\coarse \backslash \T_\fine} \eta_\coarse(T,u_\coarse)^2 \right)^{1/2}
 \le \eps + \Cdrl \, \eta_\coarse.
\end{align*}
As $\eps \to 0$, we prove~\eqref{eq:reliable}.

To prove the  weak efficiency~\eqref{eq:efficient}, we employ the inverse estimates~\eqref{eq1:invest}--\eqref{eq:invest} from Theorem~\ref{theorem:invest}. 
Since $0<s\leq 1/2$ with $s \neq 1/4$, this yields that  
\begin{align*}
 \eta_\coarse^2
 &= \norm{h_\coarse^s \big( f - (-\Delta)^s u_\coarse \big)}_{L^2(\Omega)}^2
 = \norm{ h_\coarse^s \, (-\Delta)^s( u -  u_\coarse )}_{L^2(\Omega)}^2
 \\& 
 \lesssim \norm{h_\coarse^s \, (-\Delta)^s J_\bullet( u -  u_\coarse )}_{L^2(\Omega)}^2
 + \norm{ h_\coarse^s \, (-\Delta)^s (1-J_\bullet)( u -  u_\coarse )}_{L^2(\Omega)}^2
 \\& 
 \lesssim \norm{J_\bullet( u -  u_\coarse )}_{\widetilde H^s(\Omega)}^2
 + \norm{(1-J_\bullet)( u -  u_\coarse )}_{\widetilde H^s(\Omega)}^2
 + \sum_{T \in \TT_\bullet} h_\bullet(T)^{2s} \, \norm{(1-J_\bullet)( u -  u_\coarse )}_{H^{2s}(\patch_\bullet(T))}^2.
\end{align*}
First, Lemma~\ref{lemma:sz} (\ref{item:lemma:sz-i}) guarantees that
\begin{align*}
 \norm{J_\bullet( u -  u_\coarse )}_{\widetilde H^s(\Omega)}
 + \norm{(1-J_\bullet)( u -  u_\coarse )}_{\widetilde H^s(\Omega)}
 \lesssim \norm{u -  u_\coarse}_{\widetilde H^s(\Omega)}.
 %\lesssim \norm{(1 - J_\coarse) u}_{\widetilde H^s(\Omega)}.
\end{align*}
% Moreover,~\cite[Section~E]{kop13} proves that
% \begin{align*}
%  \norm{(1 - J_\coarse) u}_{\widetilde H^s(\Omega)}
%  \lesssim \min_{v_\coarse \in S^{1,1}_0(\TT_\bullet)} \norm{h^{1-s}\nabla(u-v_\coarse)}_{L^2(\Omega)}
%  \le \norm{h^{1-s}\nabla(u-u_\coarse)}_{L^2(\Omega)}.
% \end{align*}
Next, we observe that $0 < 2s < 1$ with $2 s \ne 1/2$. 
Lemma~\ref{lemma:sz} (\ref{item:lemma:sz-ii}) and an 
interpolation argument
yield that
\begin{align*}
 h_\bullet(T)^s \norm{(1-J_\bullet)( u -  u_\coarse )}_{H^{2s}(\patch_\bullet(T))}
 \lesssim h_\bullet(T)^{1/2} \, \norm{u - u_\coarse}_{H^{1/2+s}(\patch_\bullet^2(T))}.
\end{align*}
% Since $J_\bullet$ is a projection, one can bootstrap the latter estimate to see that
% \begin{align*}
%  \norm{(1-J_\bullet)( u -  u_\coarse )}_{H^{2s}(\patch_\bullet(T))}
%  =\norm{(1-J_\bullet)^2( u -  u_\coarse )}_{H^{2s}(\patch_\bullet(T))}
%  \lesssim h_\bullet(T)^{1-2s} \, \norm{(1-J_\bullet)(u - u_\coarse)}_{H^1(\patch_\bullet^2(T))}.
% \end{align*}
% Another application of Lemma~\ref{lemma:sz}, (\ref{item:lemma:sz-ii}) yields that
% \begin{align*}
%  \norm{(1-J_\bullet)(u - u_\coarse)}_{H^1(\patch_\bullet^2(T))}
%  \lesssim \norm{\nabla(u - u_\coarse)}_{L^2(\patch_\bullet^3(T))}.
% \end{align*}
% Altogether, we thus see that
% \begin{align*}
%  h_\bullet(T)^s \, \norm{(1-J_\bullet)( u -  u_\coarse )}_{H^{2s}(\patch_\bullet(T))}
%  \lesssim h_\bullet(T)^{1-s} \,  \norm{\nabla(u - u_\coarse)}_{L^2(\patch_\bullet^3(T))}
%  \simeq \norm{h_\bullet^{1-s} \, \nabla(u - u_\coarse)}_{L^2(\patch_\bullet^3(T))}
% \end{align*}
% and hence
% \begin{align*}
%   &\sum_{T \in \TT_\bullet} h_\bullet(T)^s \, \norm{(1-J_\bullet)( u -  u_\coarse )}_{H^{2s}(\patch_\bullet(T))}
%   \simeq \left(\sum_{T \in \TT_\bullet} h_\bullet(T)^{2s} \, \norm{(1-J_\bullet)( u -  u_\coarse )}_{H^{2s}(\patch_\bullet(T))}^2 \right)^{1/2}
%   \\& \qquad
%   \lesssim \left(\sum_{T \in \TT_\bullet} \norm{h_\bullet^{1-s} \, \nabla(u - u_\coarse)}_{L^2(\patch_\bullet^3(T))}^2 \right)^{1/2}
%   \simeq  \norm{h_\bullet^{1-s} \, \nabla(u - u_\coarse)}_{L^2(\Omega)}.
% \end{align*}
Since newest vertex bisection is used to generate $\T_\coarse$, only a finite number of shapes of patches can occur. 
In particular, the hidden constant in the last estimate does not depend on $\T_\coarse$, but only on $\T_0$.
This concludes the proof.
\qed

%==========================================================================================
\subsection{Proof of Theorem~\ref{theorem:algorithm}}
%==========================================================================================
We note that \emph{reduction} from~\cite{CFPP14} is a consequence of conformity and~\eqref{axiom:stability}--\eqref{axiom:reduction}, since
\begin{align*}
 \left( \sum_{T \in \TT_\fine \backslash \TT_\coarse} \eta_\fine(T,v_\fine)^2 \right)^{1/2}
 &\reff{axiom:stability}\le  \left( \sum_{T \in \TT_\fine \backslash \TT_\coarse} \eta_\fine(T,v_\coarse)^2 \right)^{1/2}
 + \Cstab \, \norm{v_\fine - v_\coarse}_{\widetilde{H}^s(\Omega)} 
 \\&
 \reff{axiom:reduction}\le \qred \, \left( \sum_{T \in \TT_\coarse \backslash \TT_\fine} \eta_\coarse(T,v_\coarse)^2 \right)^{1/2}
 + \Cstab \, \norm{v_\fine - v_\coarse}_{\widetilde{H}^s(\Omega)}.
\end{align*} 
Therefore, Theorem~\ref{theorem:algorithm} immediately follows from~\cite[Theorem~4.1]{CFPP14}.\qed

\section{Proof of Theorem~\ref{theorem:invest}}
\label{sec:invest}

%==========================================================================================
\subsection{Proof of Lemma~\ref{lem:blowup}}
%==========================================================================================

For $x \in T$, we split the fractional Laplacian into a principal value part and a smoother, integrable part 
\begin{align}
\label{eq:10}
C(d,s)^{-1}(-\Delta)^s u_\ell(x) =  
\text{P.V.}\, \int_{B_{\dist(x,\partial T)}(x)}\frac{u_\ell(x)-u_\ell(y)}{\abs{x-y}^{d+2s}}dy +
\int_{\R^d\backslash B_{\dist(x,\partial T)}(x)}\frac{u_\ell(x)-u_\ell(y)}{\abs{x-y}^{d+2s}}dy.
\end{align}
Using polar coordinates $y = x+r \nu$, $\nu \in S^{d-1}$, where $S^{d-1}$ is the $(d-1)$-dimensional
unit sphere, and exploiting that $u_\ell|_T \in \mathcal{P}^1(T)$, we may compute the principal 
value part
\begin{align}\label{eq:pvcompute}
\text{P.V.}\, \int_{B_{\dist(x,\partial T)}(x)}\frac{u_\ell(x)-u_\ell(y)}{\abs{x-y}^{d+2s}}dy 
&= \lim_{\varepsilon \rightarrow 0} 
\int_{B_{\dist(x,\partial T)}(x)\backslash B_{\varepsilon}(x)} \frac{u_\ell(x)-u_\ell(y)}{\abs{x-y}^{d+2s}}dy 
\nonumber \\ &= 
\lim_{\varepsilon \rightarrow 0}\nabla u_\ell|_T \cdot \int_{B_{\dist(x,\partial T)}(x)\backslash B_{\varepsilon}(x)} \frac{x-y}{\abs{x-y}^{d+2s}}dy\nonumber \\
&= \lim_{\varepsilon \rightarrow 0} \nabla u_\ell|_T \cdot
\int_{\nu \in S^{d-1}} \int_{r = \epsilon }^{\dist(x,\partial T)}r^{-2s}\nu \, dr d\nu = 0,
\end{align}
where the last equality follows from interchanging the integration in $\nu$ and $r$. 
For the second part in the decomposition of $(-\Delta)^s u_\ell$ 
in (\ref{eq:10}) a similar computation using
the Lipschitz continuity of $u_\ell$ provides
\begin{align}\label{eq:estLipcont}
\left| 
\int_{\R^d\backslash B_{\dist(x,\partial T)}(x)}\frac{u_\ell(x)-u_\ell(y)}{\abs{x-y}^{d+2s}}dy
\right|
&\lesssim \norm{u_\ell}_{W^{1,\infty}(\Omega)}
\int_{B_{\dist(x,\partial T)}(x)^c} \frac{\abs{x-y}}{\abs{x-y}^{d+2s}}dy\nonumber \\
&= \norm{u_\ell}_{W^{1,\infty}(\Omega)}
\int_{\nu \in S^{d-1}} 
\int_{r = \dist(x,\partial T) }^{\operatorname{diam}\Omega}r^{-2s} 
drd \nu \nonumber\\
&\lesssim \norm{u_\ell}_{W^{1,\infty}(\Omega)}(1+\dist(x,\partial T)^{1-2s}).
\end{align}
Since 
$\dist(x,\partial T)^{\beta+1-2s} = (\omega_\ell|_T)^{\beta+1-2s}$ 
is square-integrable in view of $\beta > 2s-3/2$, Lemma~\ref{lem:blowup}
follows. 
\qed

%==========================================================================================
\subsection{Proof of Theorem~\ref{theorem:invest}}
%==========================================================================================
As in \cite{AFFKMP17}, which provides inverse estimates for the classical boundary integral
operators, the main idea of the proof of Theorem~\ref{theorem:invest} is a
splitting  of the operator into a smoother far-field part and a localized near-field part. 
The near-field and the far-field are treated separately in the following subsections.\\

%\begin{proof}[Proof of Theorem~\ref{theorem:invest}]
The proofs for both statements \eqref{eq1:invest}--\eqref{eq:invest} are fairly similar 
and only differ in the use of the inverse inequality of Lemma~\ref{lem:nearfield} below.

Let $v \in \widetilde H^{2s}(\Omega)$.
We start by a localization and a splitting into near-field and far-field.
The $L^2$-norm in the error estimator can be written as a sum over all elements
\begin{align*}
 \norm{\widetilde h_\ell^{s}(-\Delta)^s v}_{L^2(\Omega)}^2 = 
 \sum_{T \in \T_\ell}\norm{\widetilde h_\ell^{s}(-\Delta)^s v}_{L^2(T)}^2.
\end{align*}
For any constant $c_T \in \R$, we have that $(-\Delta)^sc_T \equiv 0$, which implies 
$(-\Delta)^s v = (-\Delta)^s(v-c_T)$. 

Due to the nonlocality of the fractional Laplacian, we need to split the operator into two contributions, a
localized near-field part and a smoother far-field part. 
To this end, we select, for each element $T \in \T_\ell$, a cut-off function 
$\chi_T \in C_0^{\infty}(\R^d)$ with the following properties: 
%The $\gamma$-shape regularity of the mesh $\T_\ell$ allows us 
i) $\operatorname{supp} \chi_T \subset \Omega_\ell(T)$, 
where $\Omega_\ell(T)$ is the patch of $T$ 
defined in \eqref{eq:patch};
ii) there is a set $B'$ with $T \subset B' \subset \Omega_\ell(T)$ and
$\dist(T,\partial B') \sim h_\ell(T)$; \linebreak
iii) 
$\chi_T \equiv 1$ in $B'$; 
iv)
$\norm{\chi_T}_{W^{1,\infty}(\Omega_\ell(T))} \leq C h_\ell(T)^{-1}$ 
as well as $\dist(\supp \chi_T,\partial \Omega_\ell(T)) > c h_\ell(T)$ with constants $C$, $c>0$ 
depending only on the $\gamma$-shape regularity of $\T_\ell$ and $d$.  

Then, for each element $T \in \T_\ell$, we have the splitting 
$(-\Delta)^s(v-c_T) = (-\Delta)^s(\chi_T(v-c_T))+ \linebreak (-\Delta)^s ((1-\chi_T)(v-c_T))$
into the near-field $v_{\rm near}^T$ and a far-field $v_{\rm far}^T$ given by 
\begin{align}\label{eq:nearfieldDef}
v_{\rm near}^T &:= (-\Delta)^s(\chi_T(v-c_T)),\\
\label{eq:farfieldDef}
 v_{\rm far}^T &:= (-\Delta)^s((1-\chi_T)(v-c_T)).
\end{align}
We choose 
\begin{align}
\label{eq:c_T}
c_T:= 
\begin{cases}
0 & \text{ if $\overline{T}\cap \partial \Omega \neq \emptyset$}, \\
\frac{1}{\abs{\Omega_\ell(T)}}\int_{\Omega_\ell(T)}v &\text{ otherwise.}
\end{cases}
\end{align}
If $v=v_\ell \in S^{1,1}_0(\T_\ell)$, we write $v_{\rm near,\ell}^T$, $v_{\rm near,\ell}^T$ 
for the near-field and far-field to emphasize that the fields are related to discrete functions. 
For the near-field for $v \in \widetilde{H}^{2s}(\Omega)$ with $0<s\leq 1/2$ and $s \neq 1/4$, 
Lemma~\ref{lem:nearfield} below provides the estimate
\begin{align}
\label{eq:20}
 \sum_{T \in \T_\ell} \norm{ h_\ell^{s} v_{\rm near}^T}_{L^2(T)}^2 
 \lesssim \norm{v}^2_{H^s(\Omega)}  + \sum_{T\in\T_\ell }h_\ell(T)^{2s}\norm{ v}^2_{H^{2s}(\Omega_\ell(T))},
\end{align}
and for discrete $v_\ell \in S^{1,1}_0(\T_\ell)$ and $0<s<1$
\begin{align}
\label{eq:30}
  \sum_{T \in \T_\ell} \norm{ \widetilde h_\ell^{s} v_{\rm near,\ell}^T}_{L^2(T)}^2 
 \lesssim \norm{v_\ell}^2_{H^s(\Omega)}.
\end{align}
For any $v \in \widetilde{H}^s(\Omega)$, Lemma~\ref{lem:farfield} gives
\begin{align}
\label{eq:40}
\sum_{T \in \T_\ell}\norm{ \widetilde h_\ell^{s} v_{\rm far}^T}_{L^2(T)}^2 \lesssim \norm{v}_{\widetilde H^s(\Omega)}^2. 
\end{align}
Combining (\ref{eq:20}), (\ref{eq:40}), we prove the inverse inequality 
(\ref{eq1:invest}); the estimate 
(\ref{eq:invest}) is obtained from 
(\ref{eq:30}) and (\ref{eq:40}).  
\qed%\end{proof}

\subsubsection{The near-field}

In this subsection, we treat the near-field $v_{\rm near}^T := (-\Delta)^s(\chi_T(v-c_T))$. We start with a Poincar\'e inequality:

\begin{lemma}
\label{lemma:poincare}
Let $c_T$ be given by \eqref{eq:c_T} and $v \in \widetilde{H}^s(\Omega)$. Then, for a constant $C > 0$ depending
solely on $\Omega$, $s$, the $\gamma$-shape regularity of $\T$, and the fact
that NVB is used, we have
\begin{align}
\label{eq:poincare} 
\|v - c_T\|_{L^2(\Omega_\ell(T))} 
\leq C h_\ell(T)^s \|v\|_{H^s(\Omega_\ell(T))}. 
\end{align}
\end{lemma}
\begin{proof}
If $\overline{T} \cap \partial\Omega = \emptyset$, then the Poincar\'e 
inequality 
\eqref{eq:poincare} is shown in 
\cite[Lemma~5.1]{AFFKMP17} and \cite{AcoBor17}.
For elements at the boundary, we have chosen $c_T = 0$. 
Due to the boundary condition satisfied by $v$, 
we have the Poincar\'e inequality 
$ h_\ell(T)^{-2s}\norm{v}_{L^2(\Omega_\ell(T))}^2 \lesssim \abs{v}_{H^s(\Omega_\ell(T))}^2$. 
To see this, one makes four observations: 
a) the power $h_\ell^{-2s}$ is obtained by scaling; 
b) $\overline{T} \cap \partial\Omega \ne \emptyset$ implies that at least one face of
$\partial\Omega_\ell(T)$ lies on $\partial\Omega$; c) only finitely many shapes are possible for 
$\Omega_\ell(T)$ since the meshes are created by NVB; d) the Poincar\'e inequality 
for the patch $\Omega_\ell(T)$ scaled to size $1$ holds with constant $O(1)$ by the compactness
of the embedding of $H^{s}$ in $L^2$. 
\end{proof}

\begin{lemma}\label{lem:nearfield}
There is a constant $C>0$ depending only on $\Omega$, $d$, $s$, 
and the $\gamma$-shape regularity of $\T_\ell$ such that the following holds: 
\begin{enumerate}[(i)]
\item
\label{item:lem:nearfield-i}
For $v \in \widetilde H^{2s}(\Omega)$ with $0 < s \leq 1/2$ and $s\neq 1/4$, let the near-field 
$v_{\rm near}^T$ be given by \eqref{eq:nearfieldDef}. Then,

 \begin{align}
\label{eq:lem:nearfield-1}
  \sum_{T \in \T_\ell} \norm{h_\ell^{s} v_{\rm near}^T}_{L^2(T)}^2 
  \leq C \Big( \norm{v}_{\widetilde H^s(\Omega)}^2 + 
  \sum_{T\in\T_\ell }h_\ell(T)^{2s}\norm{ v}^2_{H^{2s}(\Omega_\ell(T))}\Big).
\end{align}
\item
\label{item:lem:nearfield-ii}
For $v_\ell \in S^{1,1}_0(\T_\ell)$ and arbitrary $0<s<1$, the near-field $v_{\rm near,\ell}^T$ satisfies that  
 \begin{align*}
  \sum_{T \in \T_\ell} \norm{\widetilde h_\ell^{s} \, v_{\rm near,\ell}^T}_{L^2(T)}^2 
  \leq C \norm{v_\ell}_{\widetilde H^s(\Omega)}^2.
\end{align*}
\end{enumerate}
\end{lemma}

\begin{proof}
Before proving this lemma, we point out that we will frequently use 
in sums the shorthand $T' \in \Omega_\ell(T)$ to mean that the summation
is over all $T' \in \T_\ell$ satisfying $T' \subset \Omega_\ell(T)$. 

\emph{Proof of (\ref{item:lem:nearfield-i}):}
The mapping properties (\ref{eq:mapping-property}) 
of the fractional Laplacian as a pseudo-differential operator of order $2s$
as well as the additional assumption $v \in \widetilde H^{2s}(\Omega)\subset H^{2s}(\R^d)$ imply
\begin{align}\label{eq:nearfieldmapping}
h_\ell(T)^{2s} \norm{v_{\rm near}^T}_{L^2(T)}^2 
\lesssim h_\ell(T)^{2s} \|\chi_T( v - c_T)\|_{\widetilde H^{2s}(\Omega)}^2
\lesssim h_\ell(T)^{2s} \norm{\chi_T(v-c_T)}_{H^{2s}(\Omega)}^2,
\end{align}
where the last inequality follows from a Hardy inequality that is valid provided that $s \neq 1/4$, see,
e.g., \cite[Thm.~{1.4.4.4}]{Grisvard}.

For $s=1/2$, we have the local $H^1$-norm on the right-hand side
of (\ref{eq:lem:nearfield-1}), and we can directly estimate this by
$\norm{v}_{H^{1}(\Omega_\ell(T))}^2$. We note that this case
coincides with the result for the hypersingular integral operator in the  boundary element method,
\cite{AFFKMP17}.

The $H^{2s}$-norm is nonlocal for $s\neq 1/2$, but it has a localized upper-bound, cf.~\cite{Faer00,Faer02}, 
\begin{align}
\nonumber  & \norm{\chi_T(v-c_T)}_{H^{2s}(\Omega)}^2 \\ 
\nonumber 
& \leq\sum_{T'\in \T_\ell} \int_{T'} \int_{\Omega_\ell(T')} \frac{\abs{\chi_T(x)(v(x)-c_T) - 
 \chi_T(y)(v(y)-c_T)}^2}{\abs{x-y}^{d+4s}}dydx 
%\\ & \quad \mbox{}
 + \frac{C}{h_\ell(T')^{2s}}\norm{\chi_T(v-c_T)}_{L^2(T')}^2 \\
\nonumber 
 & \lesssim \sum_{T'\in \T_\ell} \int_{T'} \int_{\Omega_\ell(T')} 
\frac{\abs{\chi_T(x)(v(x)-v(y))}^2}{\abs{x-y}^{d+4s}}dydx 
+ 
\frac{\abs{(\chi_T(y)-\chi_T(x))(v(y)-c_T)}^2}{\abs{x-y}^{d+4s}}dydx \\
& \qquad 
 \mbox{}+ \frac{1}{h_\ell(T')^{2s}}\norm{\chi_T(v-c_T)}_{L^2(T')}^2 
\label{eq:nearfieldlocalizednorm1}
=: S_1 + S_2 + S_3.
\end{align}
\emph{Step 1:} (Estimate of $S_3$)
%We claim: $S_3 \leq C \|v\|^2_{\widetilde H^s(\Omega)}$. 
As $\supp \chi_T \subset \Omega_\ell(T)$, the last term sums up to a $L^2$-norm on the patch $\Omega_\ell(T)$. Since 
the size of neighboring elements differ only by a constant multiplicative factor, we can estimate 
$h_\ell(T')^{-2s} \lesssim h_\ell(T)^{-2s}$ for all $T' \in \Omega_\ell(T)$.
The Poincar\'e inequality on the patch $\Omega_\ell(T)$ given in 
Lemma~\ref{lemma:poincare} then gives 
\begin{align*}
S_3 = 
 \sum_{T'\in \T_\ell} h_\ell(T')^{-2s} \norm{\chi_T(v-c_T)}^2_{L^2(T')} \lesssim 
 h_\ell(T)^{-2s} \norm{v-c_T}_{L^2(\Omega_\ell(T))}^2 \lesssim \norm{v}_{H^s(\Omega_\ell(T))}^2.
\end{align*}
%
%which holds in the fractional norm due to scaling to a patch of size $O(1)$, 
%interpolation of the classical
%Poincar\'e inequality in $H^1_0(\Omega)$ and the trivial inequality in $L^2(\Omega)$, and scaling.}

\emph{Step 2:} (Estimate of $S_1$)
%
%The term $S_1$ in \eqref{eq:nearfieldlocalizednorm1} can be estimated by 
%\begin{align}\label{eq:nearfieldlocalizednorm2}
 %&\sum_{T'\in \T_\ell} \int_{T'} \int_{\Omega_\ell(T')} \frac{\abs{\chi_T(x)(v(x)-c_T) 
 %- \chi_T(y)(v(y)-c_T)}^2}{\abs{x-y}^{d+4s}}dydx
 %\nonumber \\
%S_1 &\lesssim 
% \sum_{T'\in \T_\ell} \int_{T'} \int_{\Omega_\ell(T')} \frac{\chi_T(x)^2(v(x)-v(y))^2 + 
% (v(y)-c_T)^2(\chi_T(x)-\chi_T(y))^2}{\abs{x-y}^{d+4s}}dydx
%\end{align}
%For the first addend in the integral, 
To estimate $S_1$, we use that $\supp \chi_T \subset \Omega_\ell(T)$. 
Therefore, 
only elements $T' \in \Omega_\ell(T)$ appear. With a hidden constant depending on the $\gamma$-shape regularity 
of $\T_\ell$, this leads to 
\begin{align}\label{eq:nearfieldlocalizednorm3}
% \sum_{T'\in \T_\ell} \int_{T'} \int_{\Omega_\ell(T')} \frac{\chi_T(x)^2(v(x)-v(y))^2}{\abs{x-y}^{d+4s}}dydx
S_1 
 &\lesssim \int_{\Omega_\ell(T)} \int_{\Omega_\ell(T)} \frac{(v(x)-v(y))^2}{\abs{x-y}^{d+4s}}dydx 
%\nonumber \\ &\quad
+ 
 \sum_{T'\in \Omega_\ell(T)\backslash T}\int_{T'} \chi_T(x)^2
 \int_{\Omega_\ell(T')\backslash\Omega_\ell(T)} \frac{(v(x)-v(y))^2}{\abs{x-y}^{d+4s}}dydx.
\end{align}
The first term is just the $H^s(\Omega_\ell(T))$-seminorm so that we 
may concentrate on the second term. 
Since $\operatorname*{dist}(\supp \chi_T,\partial \Omega_\ell(T)) > c h_\ell(T)$, 
the denominator in the second integrand can be directly estimated by powers of $h_\ell(T)$. 
For arbitrary $T'\in \Omega_\ell(T)\backslash T$, this immediately leads to 
\begin{align}\label{eq:nearfieldlocalizednorm4}
&\int_{T'} \chi_T(x)^2
 \int_{\Omega_\ell(T')\backslash\Omega_\ell(T)} \frac{(v(x)-v(y))^2}{\abs{x-y}^{d+4s}}dydx \lesssim
 \int_{T'} \chi_T(x)^2 \int_{\Omega_\ell(T')\backslash\Omega_\ell(T)}(v(x)-v(y))^2h_\ell(T)^{-d-4s}dydx \nonumber \\
  &\qquad\lesssim h_\ell(T)^{-d-4s}\left( \int_{T'} \chi_T(x)^2
  \int_{\Omega_\ell(T')\backslash\Omega_\ell(T)}(v(x)-c_T)^2dydx+\int_{T'} \chi_T(x)^2
  \int_{\Omega_\ell(T')\backslash\Omega_\ell(T)}(v(y)-c_T)^2dydx \right) \nonumber \\
  &\qquad\lesssim h_\ell(T)^{-4s}\norm{\chi_T(v -c_T)}_{L^2(T')}^2
  +h_\ell(T)^{-4s}\norm{v -c_T}_{L^2(\Omega_\ell(T')\backslash\Omega_\ell(T))}^2.
\end{align}
After summation over $T'$, the desired estimate follows from 
the Poincar\'e inequality of Lemma~\ref{lemma:poincare}. 
%The first term in \eqref{eq:nearfieldlocalizednorm3} is just the $H^{2s}$-seminorm on the patch $\Omega_\ell(T)$.

\emph{Step 3:} (Estimate of $S_2$)
%It remains to estimate the second addend in \eqref{eq:nearfieldlocalizednorm2}. 
Since the integrand vanishes 
if $T' \notin \Omega_\ell(T)$, we may split the sum as in Step~2 into a smooth part and a part defined on the patch 
$\Omega_\ell(T)$, i.e.,
\begin{align}\label{eq:nearfieldlocalizednorm5}
%& \sum_{T'\in \T_\ell}  \int_{T'}\int_{\Omega_\ell(T')} \frac{(v(y)-c_T)^2(\chi_T(x)-\chi_T(y))^2}{\abs{x-y}^{d+4s}}dydx \lesssim 
S_2 &\lesssim 
\int_{\Omega_\ell(T)}\int_{\Omega_\ell(T)}  \frac{(v(y)-c_T)^2(\chi_T(x)-\chi_T(y))^2}{\abs{x-y}^{d+4s}}dydx \nonumber \\
 & \qquad \mbox{}+ %\qquad \qquad \qquad \qquad \qquad \quad  +
  \sum_{T'\in \Omega_\ell(T)\backslash T}\int_{T'}
  \int_{\Omega_\ell(T')\backslash\Omega_\ell(T)} \frac{(v(y)-c_T)^2(\chi_T(x)-\chi_T(y))^2}{\abs{x-y}^{d+4s}}dydx
=:S_{2,1} + S_{2,2}.
\end{align}
Using the Lipschitz continuity of 
$\chi_T$, we obtain
\begin{align*}
% \int_{\Omega_\ell(T)}\int_{\Omega_\ell(T)} \frac{(v(y)-c_T)^2(\chi_T(x)-\chi_T(y))^2}{\abs{x-y}^{d+4s}}dydx 
S_{2,1}
&\lesssim
 \norm{\nabla \chi_T}_{L^{\infty}(\R^d)}^2\int_{\Omega_\ell(T)}\int_{\Omega_\ell(T)} \frac{(v(y)-c_T)^2}{\abs{x-y}^{d+4s-2}}dydx \\
 &\lesssim
 h_\ell(T)^{-2} \int_{\Omega_\ell(T)}(v(y)-c_T)^2\int_{\Omega_\ell(T)} \frac{1}{\abs{x-y}^{d+4s-2}}dxdy
\\
&
\stackrel{\text{Lem.~\ref{lemma:poincare}}}{\lesssim} 
 h_\ell(T)^{-2+2s} \|v\|^2_{H^s(\Omega_\ell(T))} 
\sup_{y \in \Omega_\ell(T)} \int_{\Omega_\ell(T)} \frac{1}{\abs{x-y}^{d+4s-2}}dx. 
\end{align*}
Since we assume $s < 1/2$, a direct calculation reveals 
\begin{align*}
 \int_{\Omega_\ell(T)} \frac{1}{\abs{x-y}^{d+4s-2}}dx 
 \lesssim \int_0^{ch_\ell(T)} r^{-4s+1} dr \lesssim h_\ell(T)^{-4s + 2},
\end{align*}
so that $S_{2,1}$ can be estimate in the required way. To estimate
$S_{2,2}$ one uses again the Lipschitz continuity of $\chi_T$, the observation 
$\operatorname*{dist}(\supp \chi_T,\partial \Omega_\ell(T)) > c h_\ell(T)$,
as well as a Poincar\'e inequality of Lemma~\ref{lemma:poincare} so that 
$S_{2,2}$ 
%the second term in \eqref{eq:nearfieldlocalizednorm5} 
can be estimated using the same arguments as 
in \eqref{eq:nearfieldlocalizednorm4}.
%Since $s<1/2$, the inner integral on the right-hand side is finite and can be estimated using polar 
%coordinates similarly as in \eqref{eq:estLipcont} by
%\begin{align*}
% \int_{\Omega_\ell(T)} \frac{1}{\abs{x-y}^{d+4s-2}}dx 
% \lesssim \int_0^{ch_\ell(T)} r^{-4s+1} dr \lesssim h_\ell(T)^{-4s + 2}.
%\end{align*}
%Together with a Poincar\'e inequality on the patch $\Omega_\ell(T)$, this leads to
%\begin{align*}
% \int_{\Omega_\ell(T)}\int_{\Omega_\ell(T)} \frac{(v(y)-c_T)^2(\chi_T(x)-\chi_T(y))^2}{\abs{x-y}^{d+4s}}dydx
%  &\lesssim  h_\ell(T)^{-4s} \norm{v-c_T}_{L^2(\Omega_\ell(T))}^2 
%  \lesssim h_\ell(T)^{-2s} \norm{v}_{H^s(\Omega_\ell(T))}^2.
%\end{align*}
Putting the estimates for $S_1$, $S_2$, and $S_3$ 
into \eqref{eq:nearfieldmapping} and summing over all elements using the 
$\gamma$-shape regularity of $\T_\ell$, we obtain
\begin{align*}
\sum_{T\in\T_\ell}h_\ell(T)^{2s} \norm{v_{\rm near}^T}_{L^2(T)}^2 \lesssim \norm{v}_{\widetilde H^{s}(\Omega)}^2 +
\sum_{T\in\T_\ell}h_\ell(T)^{2s} \norm{v}_{H^{2s}(\Omega_\ell(T))}^2 .
\end{align*}

\emph{Proof of (\ref{item:lem:nearfield-ii}) in the case $0 < s \leq 1/2$:} 
For $2s \leq 1$ and $v_\ell \in S^{1,1}_0(\T_\ell)$ we have the classical 
inverse estimate (e.g., \cite[Thm. 4.4.2]{SauterSchwab})
\begin{align}
\label{eq:classical-inverse-estimate}
 \norm{v_\ell}_{H^{2s}(\Omega_\ell(T))}\lesssim h_\ell(T)^{-s}\norm{v_\ell}_{H^s(\Omega_\ell(T))}. 
\end{align}
Hence, for $s \in (0,1/2]\setminus \{1/4\}$ and a $v_\ell \in S^{1,1}_0(\T_\ell)$, 
we may directly combine the result of (\ref{item:lem:nearfield-i}) with the 
inverse estimate  
\eqref{eq:classical-inverse-estimate}
to get the desired estimate. 
%Provided we are dealing with a discrete density $v_\ell$, a classical 
%inverse estimate on the patch (we have $2s < 1$) to get 
%\begin{align*}
% \norm{v_\ell}_{H^{2s}(\Omega_\ell(T))}\lesssim h_\ell(T)^{-s}\norm{v_\ell}_{H^s(\Omega_\ell(T))}, 
%\end{align*}
%and after summation to the second claimed inequality.
To obtain the case $s = 1/4$, we note that in the above proof of (\ref{item:lem:nearfield-i}), 
the assumption $s \ne 1/4$ entered only through the estimate \eqref{eq:nearfieldmapping}. 
However, the bound $\|\chi_T (v_\ell - c_T)\|_{\widetilde H^{2s}(\Omega)} 
\lesssim \|\chi_T (v_\ell - c_T)\|_{H^{2s}(\Omega)} $ is still valid for 
$v_\ell \in S^{1,1}_0(\T_\ell)$.

\emph{Proof of (\ref{item:lem:nearfield-ii}) in the case $1/2 < s < 1$:} 
%For the case $s>1/2$, we only treat discrete functions $v_\ell \in S^{1,1}_0(\T_\ell)$.
Here, we cannot use the  mapping properties 
 of the fractional Laplacian since $v_\ell$ may not be in $H^{2s}(\Omega)$ for $s\geq3/4$. 
 However, we can directly estimate the fractional Laplacian using that due to $s>1/2$ the function $r^{-2s+1}$ 
 vanishes at infinity, which does not hold for the case $s<1/2$. We write
  \begin{align*}
  \sum_{T \in \T_\ell} \norm{\widetilde h_\ell^{s} \, v_{\rm near,\ell}^T}_{L^2(T)}^2 = 
  \sum_{T \in \T_\ell}h_\ell(T)^{2s-2\beta} \norm{\omega_\ell^\beta(-\Delta)^s((v_\ell-c_T)\chi_T)}_{L^2(T)}^2.
\end{align*}
 The definition of the 
fractional Laplace leads to
\begin{align}\label{eq:investslarge1}
\norm{\omega_\ell^\beta(-\Delta)^s((v_\ell-c_T)\chi_T)}_{L^2(T)}^2 &= 
\int_T\omega_\ell(x)^{2\beta}\left(\text{P.V.}\int_{\R^d}\frac{(v_\ell(x)-c_T)\chi_T(x)-(v_\ell(y)-c_T)\chi_T(y)}{\abs{x-y}^{d+2s}} dy\right)^2 dx
\nonumber \\& \lesssim 
\int_T\omega_\ell(x)^{2\beta}(v_\ell(x)-c_T)^2\left(\text{P.V.}\int_{\R^d}\frac{\chi_T(x)-\chi_T(y)}{\abs{x-y}^{d+2s}} dy\right)^2 dx
\nonumber \\& \quad
+\int_T\omega_\ell(x)^{2\beta}\left(\text{P.V.}\int_{\R^d}\chi_T(y)\frac{v_\ell(x)-v_\ell(y)}{\abs{x-y}^{d+2s}} dy\right)^2 dx
=:S_{4} + S_{5}. 
%\\& \quad
%+\int_T\omega(x)^{2\beta}\left(\text{P.V.}\int_{\R^d}\frac{(v_\ell(x)-v_\ell(y))(\chi_T(x)-\chi_T(y))}{\abs{x-y}^{d+2s}} dy\right)^2 dx.
\end{align}
We treat the integrals $S_4$, $S_5$ on the right-hand side separately, starting with $S_5$. 
We split the integral $S_5$ further into two parts,
a principal value integral containing the singularity and a smoother, integrable part: 
\begin{align*}
%&\int_T\omega_\ell(x)^{2\beta}\left(\text{P.V.}\int_{\R^d}\chi_T(y)\frac{v_\ell(x)-v_\ell(y)}{\abs{x-y}^{d+2s}} dy\right)^2 dx
%\\& \quad \lesssim 
S_5 & \lesssim 
\int_T\omega_\ell(x)^{2\beta}\left(\text{P.V.}\int_{B_{\dist(x,\partial T)}(x)}\chi_T(y)\frac{v_\ell(x)-v_\ell(y)}{\abs{x-y}^{d+2s}} dy\right)^2 dx
\\&\quad\quad +
\int_T\omega_\ell(x)^{2\beta}\left(\int_{B_{\dist(x,\partial T)}(x)^c}\chi_T(y)\frac{v_\ell(x)-v_\ell(y)}{\abs{x-y}^{d+2s}} dy\right)^2 dx
=: S_{5,1} + S_{5,2}. 
\end{align*}
In fact, $S_{5,1} = 0$. To see this, we use that, for $x$, $y \in T$, 
we have $v_\ell(x)-v_\ell(y) = \nabla v_\ell|_T\cdot(x-y)$,
where $\nabla v_\ell|_T \in \R^d$ is constant. Using polar coordinates, 
the principal value integral can be computed and is equal to zero as in \eqref{eq:pvcompute}. 
To bound $S_{5,2}$ we note that 
our assumption $s>1/2$ implies that $r^{-2s+1}$ vanishes at infinity.
With the Lipschitz continuity of $v_\ell$, we estimate $S_{5,2}$ as in \eqref{eq:estLipcont} as 
\begin{align}\label{eq:estintfinal}
%&\int_T\omega_\ell(x)^{2\beta}\left(\int_{B_{\dist(x,\partial T)}(x)^c}\chi_T(y)
%\frac{v_\ell(x)-v_\ell(y)}{\abs{x-y}^{d+2s}} dy\right)^2 dx \nonumber \\
%&\qquad \lesssim 
S_{5,2} &\lesssim 
\norm{\nabla v_\ell}_{L^{\infty}(\Omega_\ell(T))}^2
\int_T\omega_\ell(x)^{2\beta}\left(\int_{B_{\dist(x,\partial T)}(x)^c}
\frac{1}{\abs{x-y}^{d+2s-1}} dy\right)^2 dx \nonumber \\
&\qquad \lesssim \norm{\nabla v_\ell}_{L^{\infty}(\Omega_\ell(T))}^2
\int_T\omega_\ell(x)^{2\beta}\left(\int_{\dist(x,\partial T)}^{\infty}
\frac{1}{r^{2s}} dy\right)^2 dx \nonumber
\\ &\qquad = \norm{\nabla v_\ell}_{L^{\infty}(\Omega_\ell(T))}^2
\int_T\omega_\ell(x)^{2\beta-4s+2}dx 
\lesssim h_\ell(T)^{2\beta-4s+2+d}\norm{\nabla v_\ell}_{L^{\infty}(\Omega_\ell(T))}^2 \nonumber \\
&\qquad\lesssim h_\ell(T)^{2\beta-2s}\norm{v_\ell}_{H^{s}(\Omega_\ell(T))}^2,
\end{align}
where the last two estimates follow from a direct computation of the integral over the distance function 
and an inverse inequality, where we used that the ratio of neighboring elements is bounded by a constant.

It remains to estimate $S_{4}$ in \eqref{eq:investslarge1}, 
which is similar to the integral above.  Since $\chi_T \equiv 1$ on $T$, we get
\begin{align*}
&\int_T\omega_\ell(x)^{2\beta}(v_\ell(x)-c_T)^2\left(\text{P.V.}\int_{\R^d}\frac{\chi_T(x)-\chi_T(y)}{\abs{x-y}^{d+2s}} dy\right)^2 dx
\\
&\quad=
\int_T\omega_\ell(x)^{2\beta}(v_\ell(x)-c_T)^2
\left(\int_{B_{\dist(x,\partial T)}(x)^c}\frac{\chi_T(x)-\chi_T(y)}{\abs{x-y}^{d+2s}} dy\right)^2 dx.
\end{align*}

Using the Lipschitz continuity of $\chi_T$, polar coordinates, and  
the Poincar\'e inequality, we therefore may estimate as in \eqref{eq:estintfinal}
\begin{align*}
&\int_T\omega_\ell(x)^{2\beta}(v_\ell(x)-c_T)^2
\left(\int_{B_{\dist(x,\partial T)}(x)^c}\frac{\chi_T(x)-\chi_T(y)}{\abs{x-y}^{d+2s}} dy\right)^2 dx 
\\ &\quad\lesssim 
\norm{\nabla \chi_T}_{L^{\infty}(\R^d)}^2
\int_T\omega_\ell(x)^{2\beta}(v_\ell(x)-c_T)^2\left(\int_{B_{\dist(x,\partial T)}(x)^c}
\frac{1}{\abs{x-y}^{d+2s-1}}dy\right)^2 dx  
\\ &\quad \lesssim h_\ell(T)^{-2}
\int_T\omega_\ell(x)^{2\beta-4s+2}(v_\ell(x)-c_T)^2 dx 
\lesssim h_\ell(T)^{2\beta-2s} \norm{v_\ell}^2_{H^s(\Omega_\ell(T))}. 
\end{align*}
Putting the bounds for $S_4$ and $S_5$ together and summing over all elements, 
leads to the claimed inverse estimate.
\end{proof}

\subsubsection{The far-field}

In order to treat the far-field part in the proof of the inverse estimate
of Theorem~\ref{theorem:invest}, we utilize the interpretation of the fractional Laplacian as the
Dirichlet-to-Neumann map for the extension problem \eqref{eq:extension}. This leads us 
to the problem of controlling second derivatives of the solution to the extension problem. 
This is done in the following Lemma~\ref{lem:CaccType}, 
which is proved with the techniques of tangential difference quotients, typical for elliptic PDEs, 
see, e.g., \cite{Evans}.

%\begin{definition}
%We call two (relatively open) axis parallel boxes $B \subset B'\subset \R^d\times \R_+$
%%with $\operatorname*{meas}(B'\cap \R^d \times \{0\}) > 0$ 
%nested, provided the union with their 
%reflections $\widetilde{B},\widetilde{B'} \subset \R^d\times \R_-$ satisfy
%$\dist\left(\,\overline{B\cup \widetilde{B}},\partial (\overline{B'\cup \widetilde{B'}})\right) > 0 $.
%For such boxes, we write $\dist(B,\partial B') 
%:= \dist\left(\,\overline{B\cup \widetilde{B}},\partial (\overline{B'\cup \widetilde{B'}})\right)$.
%\end{definition}
%
%
%\begin{lemma}\label{lem:CaccType}
%Let $T \in \T_\ell$, and $B,B'\subset \R^{d} \times \R_+$ be nested boxes satisfying
%$T \times \{0\} \subset B \subset B'$ and $\dist(B,\partial B') \sim h_\ell(T)$. 
%Let the Dirichlet data $u$ in the extension problem \eqref{eq:extension} satisfy $\supp u \cap B' = \emptyset$, and 
%let $U = U(x,y)$ be the solution to the extension problem \eqref{eq:extension}.
%Then,
%\begin{align}
% \norm{D_x(\nabla U)}_{L^2_\alpha(B)} \lesssim \frac{1}{h_\ell(T)} \norm{\nabla U}_{L^2_\alpha(B')}.
%\end{align}
%\end{lemma}

\begin{lemma}\label{lem:CaccType} 
Let $\zeta \in C^\infty_0(\R^{d+1})$ and $B \subset B' \subset \R^{d} \times \R^+$ with 
$\zeta|_B \equiv 1$ and $\supp \zeta \cap \R^d \times \R^+ \subset B'$. 
Let $U \in {\mathcal B}^1_\alpha(\R^d \times \R^+)$ satisfy \eqref{eq:extension-a} and 
$\operatorname{tr} (U \zeta)  = 0$. 
Then,
\begin{align}
 \norm{D_x(\nabla U)}_{L^2_\alpha(B)} \leq 2 \|\nabla \zeta\|_{L^\infty(\R^{d+1})}  \norm{\nabla U}_{L^2_\alpha(B')}.
\end{align}
\end{lemma}
\begin{proof}
With the unit-vector $e_{x_j}$ in the $j$-th coordinate and $\tau>0$, we define 
the difference quotient 
\begin{align*}
D_{x_j}^\tau w(x) := \frac{w(x+\tau e_{x_j})-w(x)}{\tau}.
\end{align*}
We use the test-function $V = D_{x_j}^{-\tau}(\zeta^2 D_{x_j}^\tau U)$ in the weak formulation of \eqref{eq:extension}.
Noting that $V|_{\partial(\Omega \times \R^+)} = 0$ due to the assumption 
$\operatorname*{tr}(U\zeta) = 0$, we therefore obtain
\begin{align*}
0 = \int_{\Omega \times \R^+} \mathcal{Y}^{\alpha} \nabla U \cdot\nabla V  d\mathcal{Y}\, dx &= 
\int_{B'} D_{x_j}^\tau (\mathcal{Y}^\alpha \nabla U) \cdot \nabla (\zeta^2 D_{x_j}^\tau U) \,d\mathcal{Y}\, dx \\ &= 
 \int_{B'} \mathcal{Y}^\alpha D_{x_j}^\tau (\nabla U) \cdot \left(\zeta^2 \nabla D_{x_j}^\tau U + 2\zeta \nabla \zeta D_{x_j}^\tau U\right) d\mathcal{Y} \, dx\\
 & =
  \int_{B'} \mathcal{Y}^\alpha \zeta^2  D_{x_j}^\tau (\nabla U) \cdot D_{x_j}^\tau (\nabla U)\, d\mathcal{Y}\,dx + 
  \int_{B'} 2 \mathcal{Y}^\alpha\zeta \nabla\zeta \cdot D_{x_j}^\tau (\nabla U)  D_{x_j}^\tau U\, d\mathcal{Y}\, dx. 
\end{align*}
Therefore, with Young's inequality, we have the estimate
\begin{align*}
 \norm{\zeta D_{x_j}^\tau(\nabla U)}_{L^2_\alpha(B')}^2 &\leq
 2 \abs{  \int_{B'} \mathcal{Y}^\alpha \zeta \nabla\zeta \cdot  D_{x_j}^\tau (\nabla U) D_{x_j}^\tau U d\mathcal{Y} \, dx} \\
 &\leq \frac{1}{2}  \norm{\zeta D_{x_j}^\tau(\nabla U)}_{L^2_\alpha(B')}^2 
 + 2 \norm{\nabla \zeta}_{L^{\infty}(B')}^2\norm{D_{x_j}^\tau U}_{L^2_\alpha(B')}^2.
\end{align*}
Absorbing the first term and taking the limit $\tau \rightarrow 0$, 
we obtain the sought inequality for all second derivatives in $x$.
\end{proof}

The following Lemma~\ref{lem:farfield} provides the required estimate for the far-field.

\begin{lemma}\label{lem:farfield}
 Let $v \in \widetilde{H}^{s}(\Omega)$ and the far-field $v_{\rm far}^T$ be given by \eqref{eq:farfieldDef}. Let $\beta = 0$ for 
 $0<s<1/2$ and $\beta = s-1/2$ for $1/2<s<1$. Then, the estimate
 \begin{align}
\label{eq:lem:farfield-10}
\sum_{T\in \T_\ell}\norm{\widetilde h_\ell^s \; v_{\rm far}^T}_{L^2(T)}^2 \leq C \norm{v}_{\widetilde H^s(\Omega)}^2
\end{align}
holds with a constant depending only on $\Omega$, $d$, $s$, and the $\gamma$-shape regularity of $\T_\ell$.
\end{lemma}

\begin{proof}
Since $\beta \ge 0$, we have $\widetilde h^s \lesssim h^s$ so that it suffices to show the estimate
\eqref{eq:lem:farfield-10} with $\widetilde h^s$ replaced with $h^s$. 
%Since in the case $1/2 \leq s < 1$ we have $\omega_\ell^{\beta} \lesssim h_\ell^\beta$ in the case $1/2 \leq s< 1$ in 
%$\widetilde h_\ell^s = h_\ell^{s-\beta}\omega_\ell^\beta$
%can be estimated by $h_\ell(T)^\beta$ on $T$, it suffices to 
%show the case $\beta = 0$.

Since the fractional Laplacian is the Dirichlet-to-Neumann map 
for the extension problem \eqref{eq:extension}, we need to control the 
Neumann data of the extension problem, i.e.,
the trace of $z(x,\mathcal{Y}):=\mathcal{Y}^\alpha \partial_\mathcal{Y} U(x,\mathcal{Y})$ at $\mathcal{Y}=0$, 
where $U$ is the solution of \eqref{eq:extension} with boundary data $u = (1-\chi_T)v + \chi_T c_T$.
We note that the definition of the cut-off function $\chi_T$ and the constant $c_T$ 
(in particular: $c_T = 0$ if $\overline{T} \cap \partial \Omega \neq \emptyset$) imply
$v \in \widetilde{H}^s$. Moreover, $(-\Delta)^s c_T \equiv 0$ implies that 
$v_{\rm far}^T = (-\Delta)^s u$.

Given $T \in \T_\ell$, the $\gamma$-shape regularity of $\TT_\ell$ implies
the existence of sets $T \subset B_0 \subset B_0^\prime \subset \Omega_\ell(T)$ with 
$\operatorname{dist}(B_0,\partial B_0^\prime) \sim h_\ell(T)$ and implied constant
depending solely on the $\gamma$-shape regularity. Additionally, we may require
that $B_0^\prime \subset \{\chi_T \equiv 1\}$. Setting 
$B:= B_0 \times (0,h_\ell(T))$, 
$B':= B_0' \times (0,2 h_\ell(T))$, we may select a cut-off function 
$\zeta \in C^\infty_0(\R^{d+1})$ with $\zeta \equiv 1$ on $B$, 
$\supp \zeta \cap \R^{d} \times \R^+ \subset B'$ and $\|\nabla \zeta\|_{L^\infty(\R^{d+1})} \lesssim h_\ell(T)^{-1}$, 
where again the implied constant depends solely on the $\gamma$-shape regularity of $\T_\ell$. 
%implies the existence of 
%two nested subsets $B \subset B' \subset \R^{d+1}$ with 
%$T \times \{0\} \subset B \subset B' \subset \Omega_\ell(T) \times [0,h_\ell(T))$ 
%and $\dist(B,\partial B') \sim h_\ell(T)$ as well as $B'\cap (\R^d \times \{0\}) \subset 
%\{x\in \R^d: \chi_T(x) = 1\} \times \{0\}$.
%Moreover, we can define a cut-off function with $\zeta \equiv 1$ on $T \times \{0\}$, $\supp \zeta \subset B$ and
%$\norm{\nabla \zeta}_{L^{\infty}(B)} \lesssim h_\ell(T)^{-1}$. 
The multiplicative trace inequality from \cite[Lemma~3.7]{KarMel18} applied with $-\alpha$ provides
\begin{align*}
 \norm{z(\cdot,0)}_{L^2(T)} &= \norm{\zeta(\cdot,0)z(\cdot,0)}_{L^2(T)} 
 \lesssim \norm{\zeta z}_{L^2_{-\alpha}(B)}+
 \norm{\zeta z}_{L^2_{-\alpha}(B)}^{(1+\alpha)/2}
 \norm{\partial_\mathcal{Y}(\zeta z)}_{L^2_{-\alpha}(B)}^{(1-\alpha)/2}
\\&\lesssim
 h_\ell(T)^{-(1-\alpha)/2}\norm{z}_{L^2_{-\alpha}(B)}+
 \norm{z}_{L^2_{-\alpha}(B)}^{(1+\alpha)/2}
 \norm{\partial_\mathcal{Y} z}_{L^2_{-\alpha}(B)}^{(1-\alpha)/2}.
\end{align*} 
The equation $-\div(\mathcal{Y}^\alpha \nabla U) = 0$ implies that $-\partial_\mathcal{Y} z = \mathcal{Y}^\alpha \Delta_x U$, so we obtain
with $s = (1-\alpha)/2$
\begin{align*}
 \norm{z(\cdot,0)}_{L^2(T)} &\lesssim h_\ell(T)^{-s}
 \norm{\mathcal{Y}^\alpha \partial_\mathcal{Y} U}_{L^2_{-\alpha}(B)}+
 \norm{\mathcal{Y}^\alpha \partial_\mathcal{Y} U}_{L^2_{-\alpha}(B)}^{(1+\alpha)/2}
 \norm{\mathcal{Y}^\alpha \Delta_x U}_{L^2_{-\alpha}(B)}^{(1-\alpha)/2} \\
 &=
 h_\ell(T)^{-s}
 \norm{\partial_\mathcal{Y} U}_{L^2_{\alpha}(B)}+
 \norm{\partial_\mathcal{Y} U}_{L^2_{\alpha}(B)}^{(1+\alpha)/2}
 \norm{\Delta_x U}_{L^2_{\alpha}(B)}^{(1-\alpha)/2}.
\end{align*}
It remains to control $\Delta_x U$ in the weighted $L^2$-norm. This is done with Lemma~\ref{lem:CaccType}, and we get
\begin{align*}
 \norm{z(\cdot,0)}_{L^2(T)} &\lesssim h_\ell(T)^{-s}\norm{\partial_\mathcal{Y} U}_{L^2_\alpha(B)}+
 h_\ell(T)^{-s}\norm{\partial_\mathcal{Y} U}_{L^2_\alpha(B)}^{(1-\alpha)/2}
 \norm{\nabla U}_{L^2_\alpha(B)}^{(1+\alpha)/2}
 \\ &\lesssim
   h_\ell(T)^{-s}\norm{\nabla U}_{L^2_\alpha(\Omega_\ell(T)\times(0,h_\ell(T)))}.
\end{align*}
With a standard {\sl a priori} estimate, the energy norm on the right-hand side can be estimated by the Dirichlet data 
$(1-\chi_T)v +\chi_T c_T$, and this finally implies
\begin{align*}
 \sum_{T\in \T_\ell}  h_\ell(T)^{2s} \norm{v_{\rm far}^T}_{L^2(T)}^2 &\lesssim \norm{\nabla U}^2_{L^2_\alpha(\Omega\times \R^+)}
\\&\lesssim \norm{(1-\chi_T)v +\chi_T c_T}_{\widetilde H^s(\Omega)}^2 \lesssim \norm{v}^2_{\widetilde H^s(\Omega)}+
\norm{\chi_T(v- c_T)}_{\widetilde H^s(\Omega)}^2
\lesssim \norm{v}^2_{\widetilde H^s(\Omega)},
\end{align*}
where the last estimate can be found in the proof of Lemma~\ref{lem:nearfield}.
This proves the estimate for the far-field.
\end{proof}

\section{Numerics}
\label{sec:numerics}

We illustrate our theoretical results of the previous sections with some numerical examples in two 
dimensions. For more numerical examples about adaptive methods for fractional diffusion, we refer to 
\cite{AinGlu17}.

\subsection{Implementational Issues}
We implement the crucial steps SOLVE (step (\ref{item:alg-i})) and ESTIMATE (step (\ref{item:alg-ii})) in Algorithm~\ref{algorithm} in MATLAB R2018a as follows:

\begin{itemize}
 \item SOLVE: To obtain the lowest-order discrete solution $u_\ell \in S^{1,1}_0(\T_\ell)$, 
 we use the existing MATLAB-code from \cite{ABB17}, where the unbounded domain 
 $\R^d$ is replaced by a (large) circle around 
 the computational domain $\Omega$. The integrals of the system matrix are computed with high precision quadrature rules. 
For domains $\Omega$ with curved boundary $\partial\Omega$, the boundary is approximated by a piecewise linear 
interpolant, which introduces a variational crime. Nevertheless, this approximation improves as the mesh size near $\partial\Omega$
is reduced.
%We note that our computations do not use curved elements. Therefore, our implementation 
%involves a variational crime 
% employ a boundary approximation on curved domains like the circle. As refinement is done towards the whole boundary,
% this approximation gets better in each refinement step. 
 
 \item ESTIMATE: For the error estimator, we need to compute the local contributions 
 $$\eta_\ell(T)= \norm{h_\ell^{s-\beta}\omega_\ell^\beta ((-\Delta)^s u_\ell - f)}_{L^2(T)},$$
 with $\beta = 0$ for $0<s\leq 1/2$ and $\beta = s-1/2$ for $1/2 < s < 1$.
 In Section~\ref{sec:mainresults}, we noted that $(-\Delta)^s u_\ell(x)$ may blow up as $x$ tends to the mesh skeleton.  
The integral on the triangle $T$ is transformed to an integral on a square by means of the Duffy transformation. 
There, the integral is approximated by a quadrature on a tensor product composite Gauss rule on meshes that are 
refined geometrically towards the edges of the square.
 %Therefore, we use a high-precision tensorized $hp$-quadrature rule (mapped to a triangle) to compute the quadratic
 %integrals in the local contributions. 
Evaluating the residual requires  
evaluating the fractional Laplacian applied to the discrete solution in each quadrature point.  
For $x \in T$, this is done with the pointwise evaluation formula from \cite[Lemma~4]{AinGlu17}
 \begin{align}\label{eq:evalformula}
 \frac{(-\Delta)^s u(x)}{C(d,s)} &= \frac{1}{d+2s-2} \int_{\partial T}\frac{\nabla u|_{T}\cdot n_y}{\abs{x-y}^{d+2s-2}}dy
 -\frac{u(x)}{2s}\int_{\partial T }\frac{(x-y)\cdot n_y}{\abs{x-y}^{d+2s}}dy \nonumber\\
 &\quad  + \sum_{T' \neq T} \left(\frac{1}{2s(d+2s-2)}
 \int_{\partial T'}\frac{\nabla u|_{T'}\cdot n_y}{\abs{x-y}^{d+2s-2}}dy-
 \frac{1}{2s}\int_{\partial T' }u(y)\frac{(x-y)\cdot n_y}{\abs{x-y}^{d+2s}}dy\right),
\end{align}
Here, only integrals over the boundary of each element appear,
which are approximated using the standard 1D-MATLAB quadrature function. 

As two contributions $\eta_\ell(T_1)$, $\eta_\ell(T_2)$ are independent for $T_1 \neq T_2$, the computation 
of the error estimator is easily parallelized, which leads to considerable speed-up in the computations.
\end{itemize}

\begin{remark}
 The computation of the error estimator is fairly expensive due to the need for an accurate quadrature rule
 for the residual. One way to circumvent this is to use a different error estimator inspired by the near-field and far-field
 splitting from Section~\ref{sec:invest}, as well as the observation that the near-field behaves like a 
 power of the distance to the skeleton. We illustrate our ideas in the one dimensional case:
As in formula  \eqref{eq:pvcompute}, 
 the principal value integral over one element $T = [x_0,x_1]$ can be 
 computed exactly for $x \in T$ as 
 \begin{align*}
 \text{P.V.}\; 
 \int_{T}\frac{u_\ell(x)-u_\ell(y)}{\abs{x-y}^{1+2s}}dy =  
 \frac{u_\ell'|_T}{1-2s}\left((x-x_0)^{1-2s} - (x_1-x)^{1-2s}\right).
 \end{align*}
 For all other elements $T' \in \T_\ell \backslash \{T\}$ a similar computation can be made. However, 
 only elements in the patch of $T$ are truly of interest, since the integrand is smooth for all other elements. 
 For the elements in the patch, again, only the terms containing $x_0,x_1$ are relevant, which leads to the 
 splitting
  \begin{align*}
  r_{\rm near}(x) &:= \frac{1}{1-2s}\sum_{T'\in \Omega_\ell(T) \backslash T}\dist(x,\partial T)^{1-2s}[u_\ell']_{T,T'} \\
  r_{\rm far}(x) &:=  \sum_{\widetilde T \notin \Omega_\ell(T)} \left(\frac{1}{2s(2s-1)}
 \int_{\partial \widetilde T}\frac{u_\ell'|_{\widetilde T}\cdot n_y}{\abs{x-y}^{2s-1}}dy-
 \frac{1}{2s}\int_{\partial  \widetilde T }u_\ell(y)\frac{(x-y)\cdot n_y}{\abs{x-y}^{1+2s}}dy\right) - f,
 \end{align*}
 where $[u_\ell']_{T,T'}$ denotes the jump of the derivative across the elements $T$ and $T'$.
Similarly as in Theorem~\ref{theorem:reliable}, the residual can then be bounded by 
  \begin{align*}
  \norm{r_\ell}_{H^{-s}(\Omega)} \lesssim  \norm{h_\ell^s r_{\rm far}}_{L^2(\Omega)} + 
\sum_{T \in \T_\ell}\sum_{T'\in \Omega_\ell(T) \backslash T}\abs{[u_\ell']_{T,T'}}h_\ell(T)^{3/2-s}.
 \end{align*}
  The advantage of this approach is that the far-field, written in terms of formula \eqref{eq:evalformula}, 
  can be cheaply computed using standard 
 Gaussian quadrature and for the near-field only the jumps across the elements need to be computed.
\eremk
\end{remark}

We present convergence plots in the energy norm, where we use that --- due to the Galerkin orthogonality
and symmetry of $a(\cdot,\cdot)$ --- 
the error can be computed as 
\begin{align}
\norm{u-u_\ell}^2_{\widetilde{H}^s(\Omega)} = \skp{f,u}_{L^2(\Omega)}-\skp{f,u_\ell}_{L^2(\Omega)}.
\end{align}

\subsection{Example 1 -- Circle}
We start with an example on the unit circle $\Omega = B_1(0)$ and choose a constant right-hand side 
$f = 2^{2s}\Gamma(1+s)^2$ with exact solution given by (see, e.g., \cite{AinGlu17,BBNOS17}) 
\begin{align*}
 u(x) = (1-\abs{x}^2)_+^s \qquad \text{with} \;\; g_+ = \max\{g,0\}. 
\end{align*}
Using polar coordinates, we can easily compute the exact energy as 
$$a(u,u) = \int_{B_1(0)}fu\, dx = 2^{2s}\Gamma(1+s)^2\frac{2\pi}{2s+2}.$$

The exact solution is smooth inside the unit circle, but non-smooth towards the whole boundary,
which is typical for solutions of our model problem \eqref{eq:modelproblem}. Therefore, we expect 
that the adaptive algorithm refines the mesh towards the whole boundary, which is indeed the case as shown in 
Figure~\ref{fig:grid}. 

\begin{figure}[ht]
%\begin{minipage}{.49\linewidth}
\includegraphics[width=0.32\textwidth]{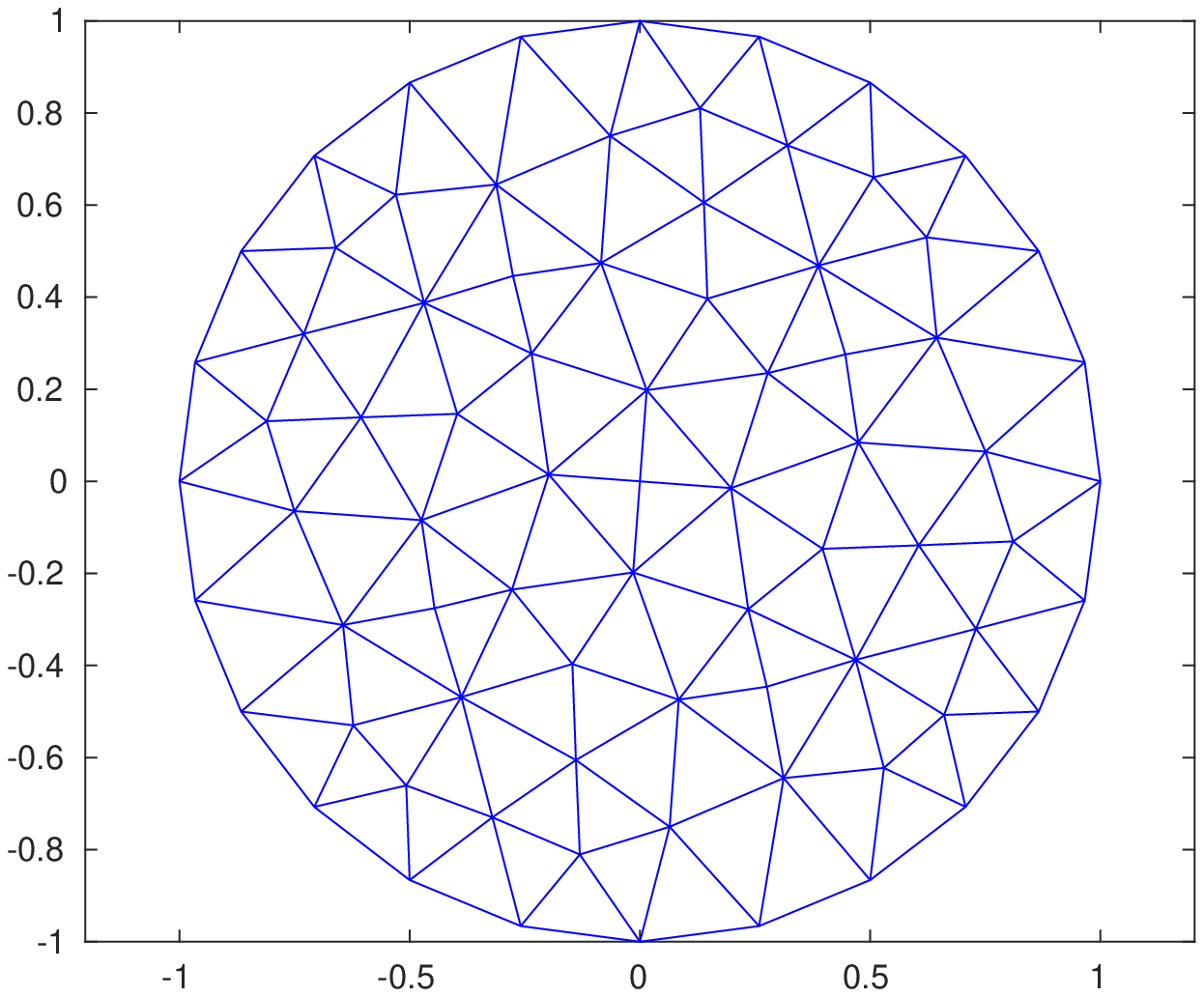}
\includegraphics[width=0.32\textwidth]{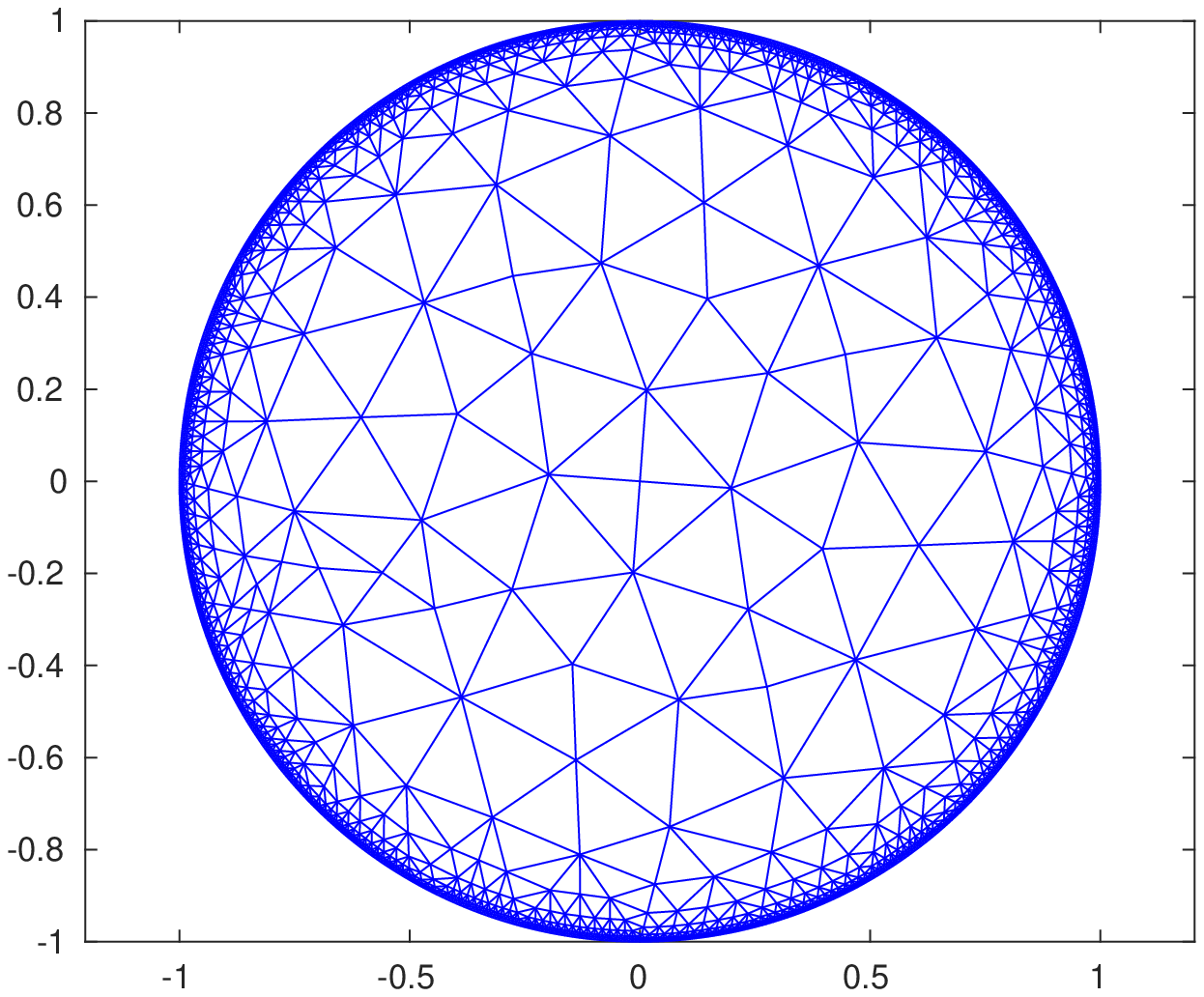}
%\end{minipage}
%\begin{minipage}{.49\linewidth}
\includegraphics[width=0.32\textwidth]{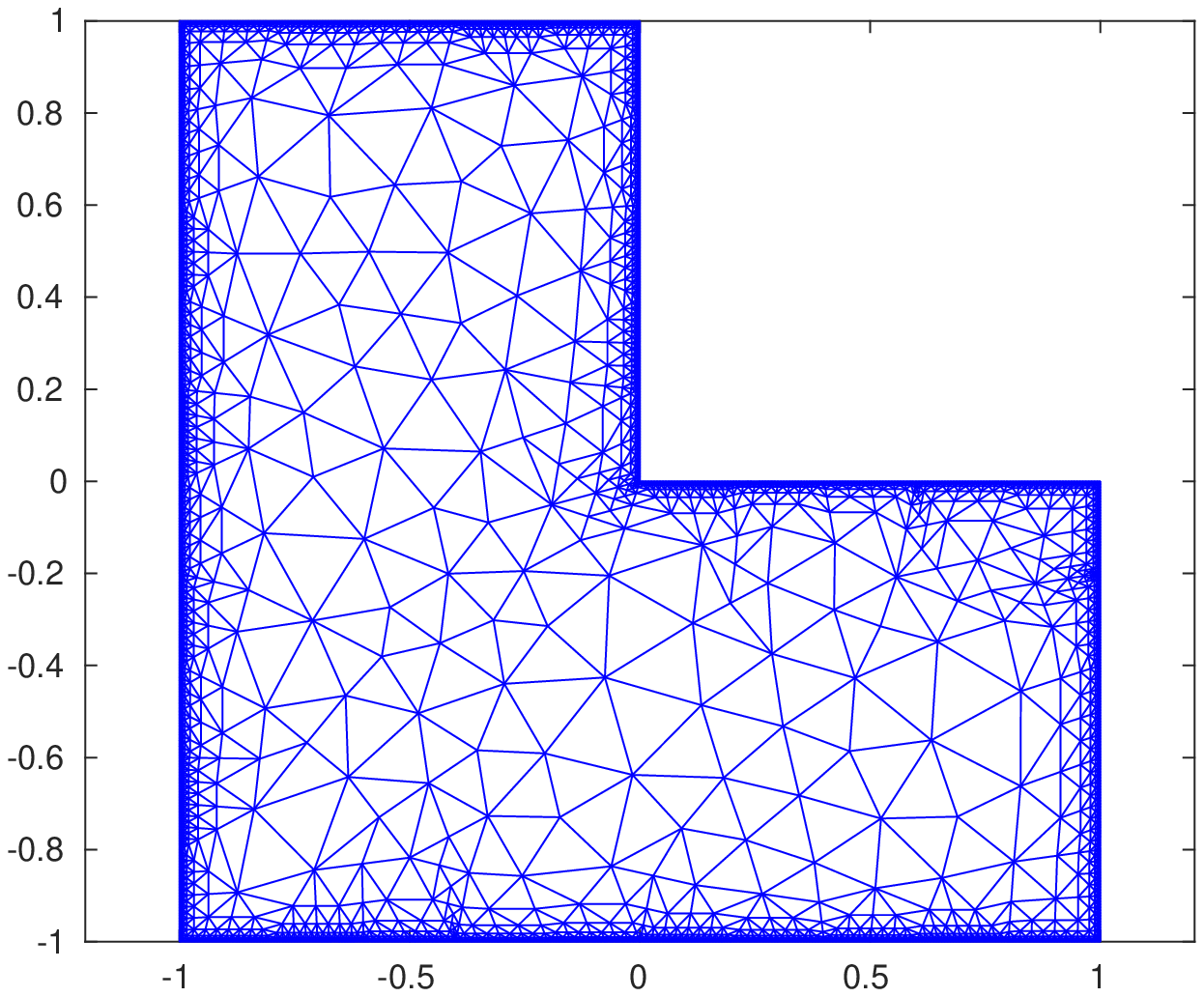}
%\end{minipage}
\centering
\caption{Adaptively generated meshes ($s=0.25$, $\theta=0.3$), left: coarse grid on the circle; middle: adaptively refined mesh on circle; 
right: L-shaped domain.}
 \label{fig:grid}
\end{figure}

Figure~\ref{fig:errorCircleConst} shows the values of the error estimator \eqref{eq:estimator} 
(brown triangles, blue squares) as well as the error (black diamonds, red stars) in the energy norm 
for uniform and adaptive mesh refinement by Algorithm~\ref{algorithm} steered by 
the estimator \eqref{eq:estimator} for $s=0.25$ (with $\theta = 0.3$) and $s=0.75$ (with $\theta = 0.5$).

\begin{figure}[ht]
\begin{minipage}{.50\linewidth}
\includegraphics{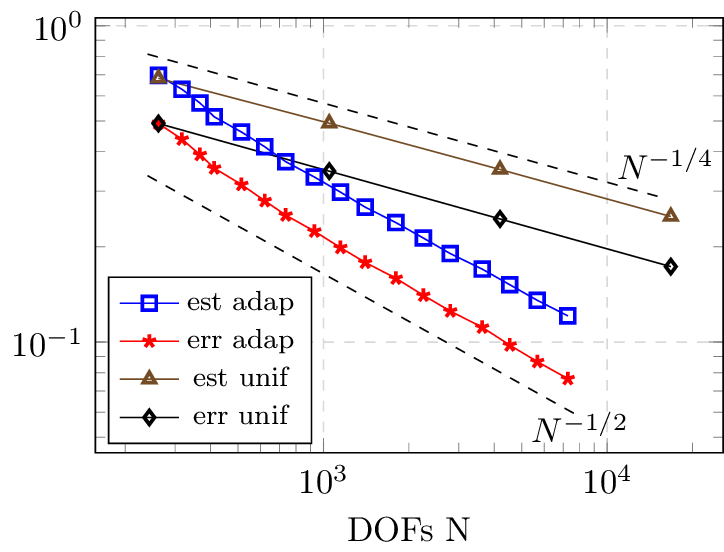}
\end{minipage}
\begin{minipage}{.49\linewidth}
\includegraphics{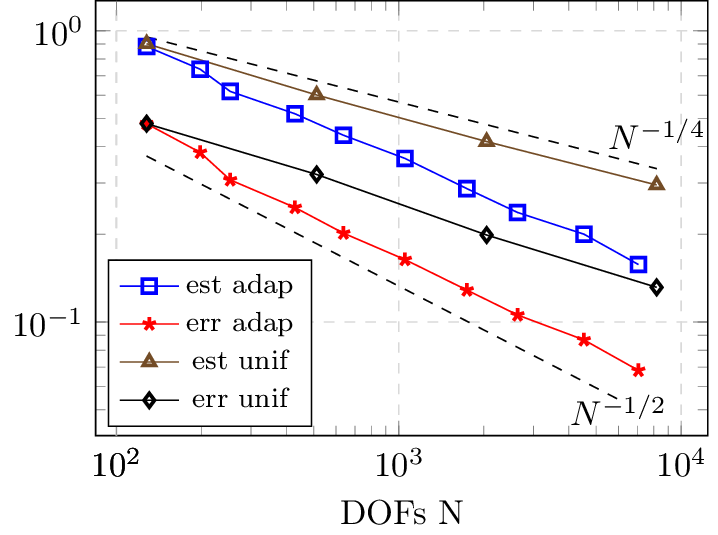}
\end{minipage}
\centering
\caption{Error and error estimator on the unit circle with constant right-hand side
for uniform and adaptive refinement; 
left: $s=0.25$, $\theta = 0.3$; 
right: $s=0.75$, $\theta = 0.5$.}
 \label{fig:errorCircleConst}
\end{figure}

As expected, uniform mesh refinement leads to a reduced rate of $N^{-1/4}$ due to the singularity of 
the solution at the boundary of the circle. However, adaptive refinement restores the optimal 
algebraic rate of $N^{-1/2}$ both for the error and the estimator as predicted by Theorem~\ref{theorem:algorithm}.
\\

In Figure~\ref{fig:errorCircleTheta}, we vary the adaptivity parameter $\theta \in (0,1]$ for fixed $s=0.25$. 
The dashed lines indicate the values of the error estimator and the solid ones the true errors in the energy norm. 
We observe 
optimal algebraic rates for all choices $\theta = 0.3,0.5,0.7$ that are smaller than 1. 
$\theta = 1$ leads to uniform refinement and therefore non-optimal convergence. 
As the error and estimator curves are not distinguishable for $\theta < 1$, this experiment suggests, 
that the condition $\theta \ll 1$ in Theorem~\ref{theorem:algorithm} is not necessary.

\begin{figure}[ht]
\centering
\includegraphics{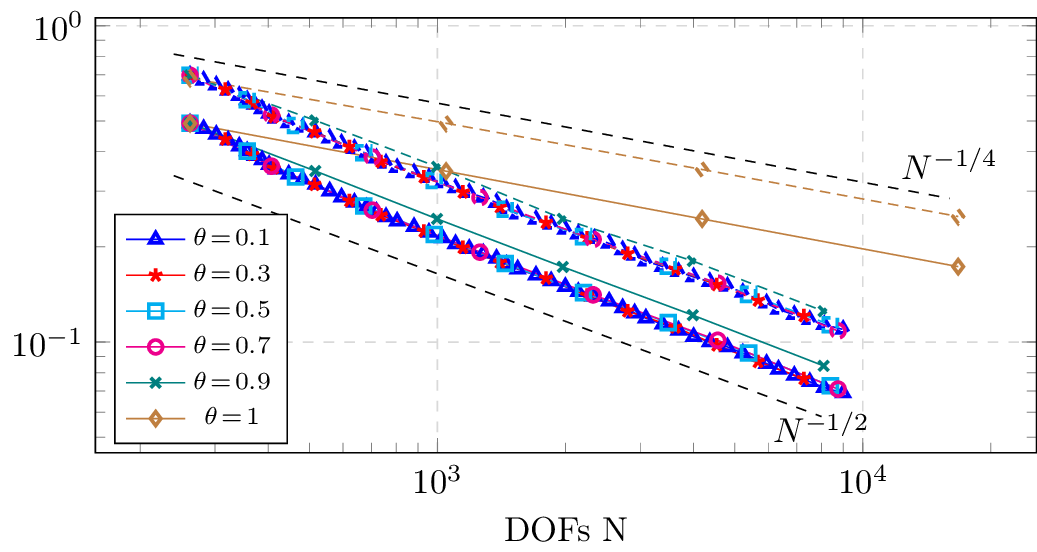}
\caption{Dependence on the adaptivity parameter $\theta$ for $s=0.25$ on the unit circle with constant right-hand side.}
 \label{fig:errorCircleTheta}
\end{figure}

\subsection{Example 2 -- L-shaped domain}
As a second example, we consider the L-shaped domain
$\Omega = (-1,1)^2\backslash [0,1)^2$ depicted in Figure~\ref{fig:grid}.
We prescribe a constant right-hand side $f \equiv 1$.
As the exact solution is unknown, we extrapolate the energy from the computed discrete energy. 

Again, the adaptive algorithm refines the meshes towards the whole boundary due to the singularity of the solution  
at the whole boundary. This is in contrast to the 
known results for the integer Laplacian ($s=1$), where only refinement towards the reentrant corner is observed.

\begin{figure}[ht]
\begin{minipage}{.50\linewidth}
\includegraphics{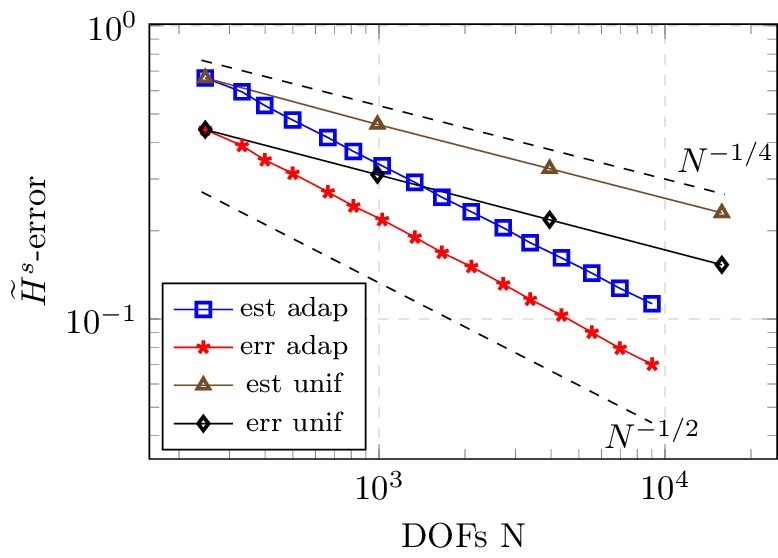}
\end{minipage}
\begin{minipage}{.49\linewidth}
\includegraphics{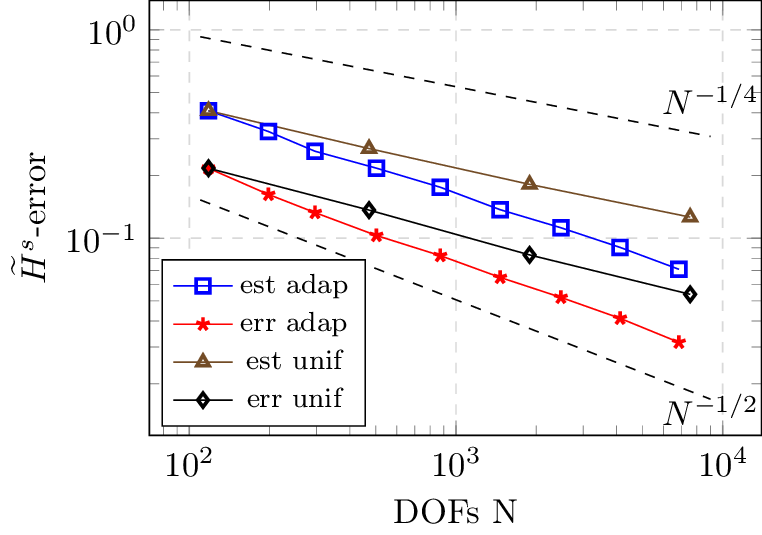}
\end{minipage}
\centering
\caption{Error and error estimator on the L-shaped domain for uniform and adaptive refinement; 
left: $s=0.25$, $\theta = 0.3$; 
right: $s=0.75$, $\theta = 0.5$.}
 \label{fig:errorLshape}
\end{figure}

Figure~\ref{fig:errorLshape} shows the convergence of the error estimator and the error for the 
cases $s=0.25$ ($\theta = 0.3$) and $s=0.75$ ($\theta = 0.5$). We observe optimal convergence rates for the adaptive method and reduced rates 
for uniform refinement just as in Example~1.

\subsection{Example 3 - Circle, discontinuous right-hand side}

As a final example, we again consider the unit circle $\Omega = B_1(0)$, but choose a discontinuous
right-hand side 
$$f(x,y) = \chi_{\{x>0\}}(x,y) = \left\{
\begin{array}{l}
  1 \quad \text{for} \, x>0 \\
 0 \quad \textrm{otherwise}
 \end{array}
 \right..$$
 For this problem, again an exact solution is known, see, e.g., \cite{AinGlu17}. The energy 
 of this solution can be computed using the Meijer G-function as 
 \begin{align*}
  \skp{f,u}_{L^2(\Omega)} = 2^{-2s} \left(\frac{\pi}{4(s+1)\Gamma(s+1)^2} - \frac{1}{\pi}
  G_{4,4}^{3,2} \Big(\begin{array}{llll} 1,&  1+s/2,&  5/2+s,&  5/2+s  \\ 
  2,&  1/2,&  1/2,&  2+s/2\end{array} \Big| -1\Big)\right).
 \end{align*}

 In Figure~\ref{fig:gridDiscontF}, an adaptively generated mesh and the discrete solution are 
 plotted. In contrast to the previous examples, the adaptive algorithm does not only refine the mesh
 at the boundary, but also along the discontinuity of $f$ at the line $x=0$. However, the refinement 
 towards the boundary tends to be stronger than towards the singularity of the right-hand side. 

\begin{figure}[ht]
%\begin{minipage}{.49\linewidth}
\includegraphics[width=0.38\textwidth]{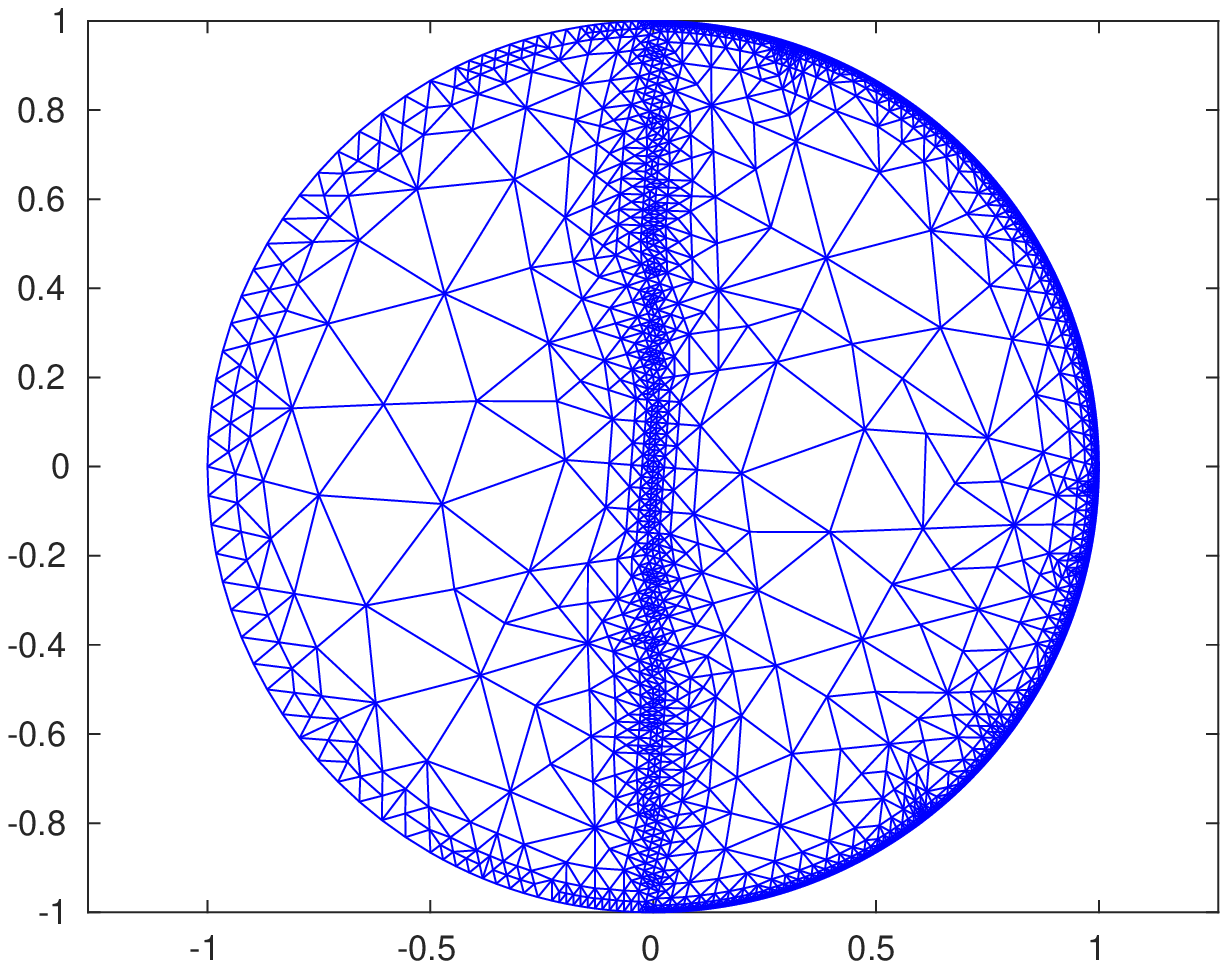}
%\end{minipage}
%\begin{minipage}{.49\linewidth}
\includegraphics[width=0.38\textwidth]{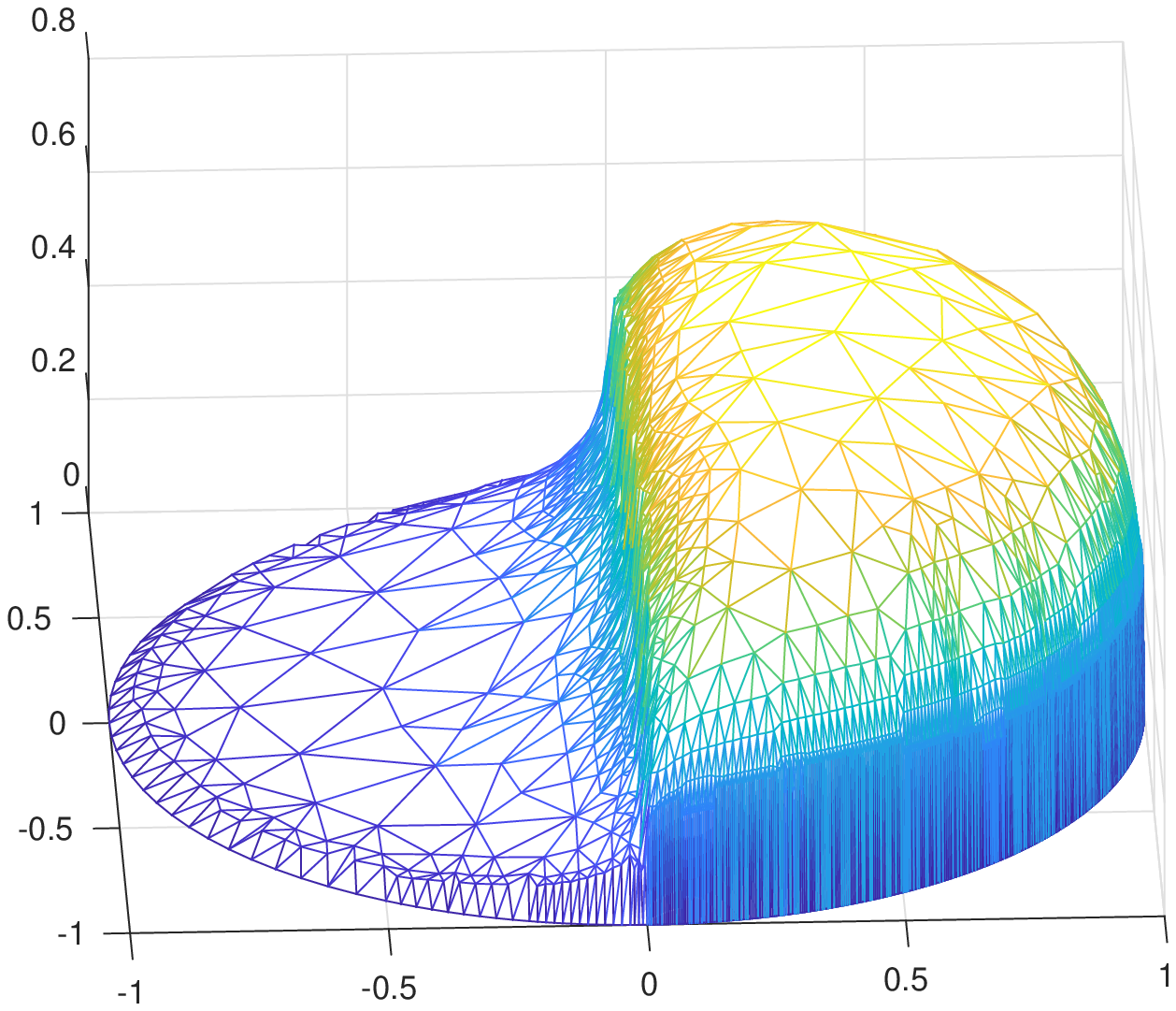}
%\end{minipage}
\centering
\caption{Adaptively generated mesh ($s=0.25$, $\theta=0.3$) for problem with discontinuous right-hand side (left) and computed 
Galerkin solution (right).}
 \label{fig:gridDiscontF}
\end{figure}

In Figure~\ref{fig:gridDiscontF}, convergence rates for the error estimator and the error 
are depicted. 
The empirical results are the same as for the previous two examples. 

\begin{figure}[ht]
\begin{minipage}{.50\linewidth}
\includegraphics{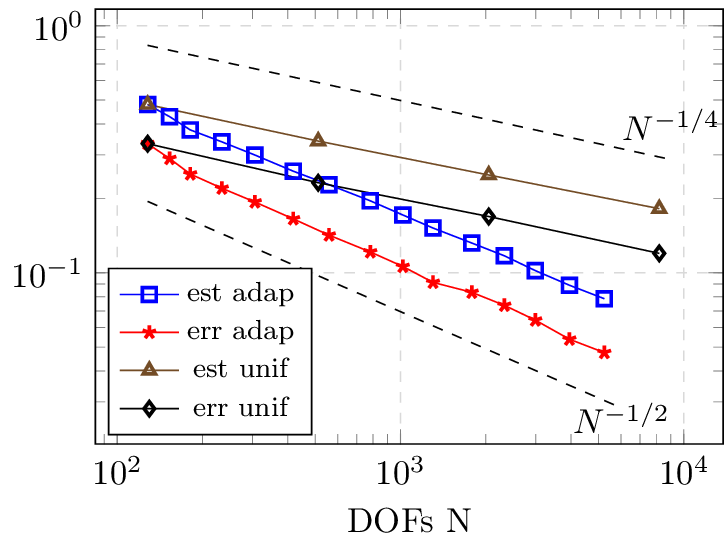}
\end{minipage}
\begin{minipage}{.49\linewidth}
\includegraphics{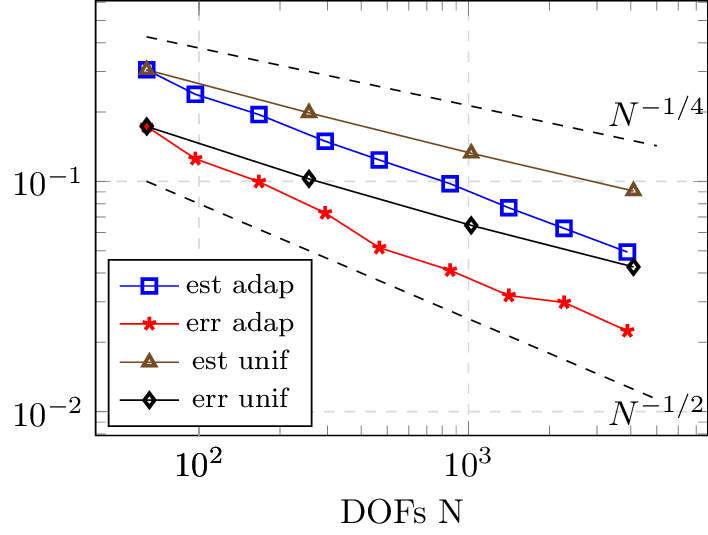}
\end{minipage}
\centering
\caption{Error and error estimator with discontinuous right-hand side 
on the unit circle for uniform and adaptive refinement; 
left: $s=0.25$, $\theta = 0.3$; 
right: $s=0.75$, $\theta = 0.5$.}
 \label{fig:errorCircleDisc}
\end{figure}

\clearpage

%\nocite{*}
\bibliography{bibliography_2}{}
\bibliographystyle{amsalpha}
\end{document}